\documentclass{amsart}
\usepackage[dvips]{epsfig}
\usepackage{graphicx}
\usepackage{latexsym}
\usepackage{amsmath}
\usepackage{amsthm}
\usepackage{amssymb}
\RequirePackage[colorlinks=true,citecolor=blue,urlcolor=blue]{hyperref}
\usepackage{natbib}
\setcitestyle{numbers,square}

\usepackage{xcolor}

\usepackage{caption}
\usepackage{subcaption}
\usepackage{float}

 \usepackage[applemac]{inputenc}

\newtheorem{lemma}{Lemma}

\newtheorem{corollary}{Corollary}

\newtheorem{proposition}{Proposition}
\newtheorem{remark}{Remark}

\newenvironment{preuve}{\vip \noindent {\it Proof}}{\hfill$\square$\vip}

\newcommand{\vip}{\vskip.2cm}

\newcommand{\COMMENTAIRE}[1]{}

\begin{document}

\title[Control the difference between Brownian motions for energy markets]{On the control of the difference between two Brownian motions: an application to energy markets modeling}

\author{Thomas Deschatre}

\address{Thomas Deschatre, CEREMADE,
Universit\'e Paris-Dauphine, Place du mar\'echal De Lattre de Tassigny
75775 Paris Cedex 16, France.}

\email{thomas.deschatre@gmail.com}

\begin{abstract} We derive a model based on the structure of dependence between a Brownian motion and its reflection according to a barrier. The structure of dependence presents two states of correlation: one of comonotonicity with a positive correlation and one of countermonotonicity with a negative correlation. This model of dependence between two Brownian motions $B^1$ and $B^2$ allows for the value of $\mathbb{P}\left(B^1_t - B^2_t \geq x\right)$ to be higher than $\frac{1}{2}$ when $x$ is close to 0, which is not the case when the dependence is modeled by a constant correlation. It can be used for risk management and option pricing in commodity energy markets. In particular, it allows to capture the asymmetry in the distribution of the difference between electricity prices and its combustible prices. 
\end{abstract}

\maketitle



\section{Introduction}

\subsection{Motivation}

One of the major issues in commodity energy markets is the pricing and hedging of multi-assets options, in particular the spread options. For instance, if we denote by $X_t$ the price of electricity at time $t$, by $Y_t$ the price of coal at time $t$, and by $H$ the heat rate (conversion factor) between the two, the income of the coal plant can be modeled by $\left(X_t - HY_t -K\right)^+$, with $x^+ = \max\left(x,0\right)$ and $K$ representing a fixed cost. To evaluate the value of this coal plant, one needs to model jointly the price of electricity and the price of coal. Because the coal is a fuel for electricity, the two prices can not be considered independent and dependence between the two needs to be modeled. For more information on spread options, the reader can refer to \citep{carmona03}.

\medskip
To model energy commodities forward prices, two different approaches exist: the first one consists in the modeling of the spot price and is equivalent to the Vasicek modeling of interest rates \citep{vasicek77}, the second one consists in the modeling of the forward curve and is similar to the a Heath-Jarrow-Morton approach \citep{heath92}. In the first approach, one way to model dependence is the use of structural models \citep{aid09, aid13, carmona14}. In structural models, electricity is a function of the residual demand and of the fuels used to produce it. Some constraints are imposed in order for the electricity price to be higher than the minimum price of its combustibles with a high probability.  An other way to model dependence is the use of co-integration between the different commodities spot prices \citep{nakajima12}. However, structural models and co-integration models are very computational costly and are not adapted for practitioners. They prefer to use the second type of models, using a forward curve. We denote by $f^i\left(t,T\right)$ the forward price of commodity $i$ at time $t$ with maturity $T$, that is of the delivery of commodity $i$ during one unit of time. The most common model for $f^i\left(t,T\right)$ is the two-factor model, see \citep{benth08} for instance. The forward price of commodity $i$, $i = 1, .., n$ is modeled by the following stochastic differential equation:
\begin{equation} \label{factosde}
df^i\left(t,T\right) = f^i\left(t,T\right) \left(\sigma^i_s e^{-\alpha^i\left(T-t\right)} dB^{s,i}_t + \sigma^i_l dB^{l,i}\right)
\end{equation}
with $B^{s,i}$ and $B^{l,i}$, $i = 1, .., n$,  $2n$ brownian motions. The dependence between the Brownian motions is usually modeled by a constant correlation matrix. In the following, we are interested only in factorial models with two commodities, electricity and one of its fuel. Marginals model (if we consider only one commodity) are really efficient and allow us to price efficiently options based on one underlying. However, dependence modeling is not satisfying because it does not capture the asymmetry in the distribution of the difference between the forward price of electricity and the one of its fuel. Furthermore, the probability for the price of electricity to be lower than the price of its fuel is closed to $\frac{1}{2}$ which is not consistent with the reality. Indeed, the fuel is used to produce the electricity.    

\medskip 
Modeling the dependence between the forward prices is equivalent to the modeling of the dependence between the Brownian motions. We consider only two Brownian motions. To capture asymmetry, it is needed to consider an other approach than the constant correlation model. A common approach to construct a pair of Brownian motions is the use of stochastic correlation. Stochastic correlation models are a generalization in a multivariate framework of stochastic volatility models, such as the Heston model \citep{heston93} where the volatility is modeled by a Cox-Ingersoll-Ross process. The matrix of volatility-correlation is stochastic and can be modeled for instance by a Wishart processes \citep{gourieroux09}. In a stochastic correlation framework, as the difference between the two Brownian motions does not follow a normal law, it is possible to capture asymmetry. However, if the stochastic correlation $\left(\rho_s\right)_{s \geq 0}$ is independent from the two Brownian motions, we have for $x \geq 0$:
\[\mathbb{P}\left(B^1_t - B^2_t \geq x\right) = \mathbb{E}\Bigl(\Phi\Bigl( \frac{-x}{\sqrt{2\int_0^t \left(1-\rho_s\right)ds}} \Bigr) \Bigr)\leq \frac{1}{2}\]
with $\Phi$ the normal cumulative distribution function. Stochastic correlation does not allow to have higher value then $\frac{1}{2}$ for $\mathbb{P}\left(B^1_t - B^2_t \geq x\right)$. An other way to construct a pair of Brownian motions is the use of a local correlation. The concept of local correlation is directly derived from the one of local volatility. In a Black and Scholes framework, the volatility is constant with the maturity and strikes which is not coherent with the implied volatilities from call and put option prices. Dupire introduces the local volatility in order to have a price model which is compatible with the volatility smiles and which is a complete market model \citep{dupire94}. Langnau introduces local correlation model which is the generalization of local volatility for a multi-dimensional framework \citep{langnau10}. A less common approach is the use of copulae. Copulae are used to model the dependence between random variables and have many applications in finance \citep{cherubini04}. Indeed, Sklar's theorem \citep{sklar59} states that modeling the distribution of a couple of random variables $\left(X,Y\right)$ is equivalent than modeling the law of $X$, the law of $Y$ and a copula function $C$ corresponding to the dependence between the two. However, use of copulae is more complicated in a continuous time framework, that is when processes are involved. In \citep{bosc12} and \citep{jaworski13}, a partial derivative equation is derived linking the copula between two Brownian motions and their local correlation function based on  the Kolmogorov forward equation. Constraints on the copula to be admissible for Brownian motions are very restrictive, especially if one want to find asymmetric copulae admissible for Brownian motions. Deschatre \citep{deschatre16} derives families of copula that are admissible for Brownian motions and asymmetric. Furthermore, he studies the range of the function $C \mapsto \mathbb{P}_C\left(B^1_t - B^2_t \geq \eta\right)$ for $\eta  > 0$ and $t > 0$ where $B^1$ and $B^2$ are two Brownian motions and $\mathbb{P}_C$ denotes the measure of probability when $C$ is the copula of $\left(B^1,B^2\right)$. Some Markovian constraints are imposed on the copula $C$. The range of this function is equal to $\left[0, 2\Phi\left(\frac{-\eta}{2\sqrt{t}}\right)\right]$ and the supremum is achieved with the copula between the Brownian motion and its reflection according to the barrier $\frac{\eta}{2}$. However, those results are not adapted to a modeling framework because of the degenerescence of the model: the Brownian motions are either correlated to 1 or to -1 depending on the value of $B^1$.

\subsection{Objectives and results} The main objective of this paper is to construct a model of dependence for solutions $f^1\left(t,T\right)$ and $f^2\left(t,T\right)$ of the stochastic differential equations \eqref{factosde}. This model of dependence must create asymmetry in the difference between the two processes.  In particular, we want to have high value for $\mathbb{P}\left(f^1\left(t,T\right) - f^2\left(t,T\right)\geq x\right)$ with $x$ close to 0. The dependence between the two processes is determined by the dependence between the Brownian motions. We reduce our problem to the construction of two Brownian motions $B^1 = \left(B^1_t\right)_{t \geq 0}$ and $B^2 = \left(B^2_t\right)_{t \geq 0}$ presenting asymmetry in their dependence and with values for $\mathbb{P}\left(B^1_t - B^2_t \geq x\right)$ higher than $\frac{1}{2}$ when $x$ is close to 0. Our model is based on the work of Deschatre \citep{deschatre16}. The value of $\mathbb{P}\left(B^1_t - B^2_t \geq \eta\right)$ for $\eta > 0$ is maximized when $B^2$ is the reflection of $B^1$ according to the barrier $\frac{\eta}{2}$. The copula between those two Brownian motions presents two states of dependence: one of comonotonicity corresponding to a correlation of 1 and one of countermonotonicity corresponding to a correlation of -1. We release these two states of dependence by allowing lower correlations in absolute value. This gives the copula of Proposition \ref{nondegenerated}. This copula is asymmetric and Proposition \ref{survival2} gives the survival function of difference between the two Brownian motions coupled with this copula:
\[\mathbb{P}\left(B^1_t - B^2_t \geq x \right) = \Phi\Bigl(\frac{-x+2\rho h}{\sqrt{2\left(1-\rho\right)t}}\Bigr) \Phi\Bigl(\frac{x-2h\left(1+\rho\right)}{\sqrt{2\left(1+\rho\right)t}}\Bigr) + \Phi\Bigl(\frac{2h-x}{\sqrt{2\left(1-\rho\right)t}}\Bigr) \Phi\Bigl(\frac{-x}{\sqrt{2\left(1+\rho\right)t}}\Bigr).\]
This model of dependence gives higher values for $\mathbb{P}\left(B^1_t - B^2_t \geq x \right)$ than the constant correlation case and than $\frac{1}{2}$ when $x$ close to 0 and for $\rho$ high enough.

\medskip
We generalize this model by allowing several reflections: it is the multi-barrier correlation model. We define two barriers $\nu$ and $\eta$ with $\nu < \eta$. We consider two independent Brownian motions $X$ and $B^Y$, and we construct the Brownian motion $Y^n$ that is correlated to $\tilde{X}^n$: 
\[Y^n = \rho\tilde{X}^n + \sqrt{1 - \rho^2}B^Y,\]
with $\tilde{X}^n$ the Brownian motion equal to $-X$ at the beginning and reflecting when $X-Y^n$ hits a two-state barrier equal to $\eta$ before the first reflection and switching from $\eta$ to $\nu$ or from $\nu$ to $\eta$ at each reflection. For a given $x \in \left[\eta,\nu\right]$ and $t >0$, Corollary \ref{convergencesurvival} states that the sequence $\mathbb{P}\left(X_t - Y^n_t \geq x\right)$ is increasing with $n$. Furthermore, the number of reflections in $\left[0,t\right]$ $N_t$ is finite almost surely, see Proposition \ref{model} (iii). We then consider the process $Y_t = Y_t^{N_t}$ which is a Brownian motion, see Proposition \ref{model} (iv), that corresponds to the case $n \rightarrow \infty$. Proposition \ref{spreaddistrib} gives the survival function of $X_t-Y_t$, which is higher than in the constant correlation case and higher than $\frac{1}{2}$ when $x \in \left[\eta,\nu\right]$ and $\rho$ is high enough. This model can be transposed to a local correlation model:
\[
\left \{
\begin{array}{l}
   dX_t = dB^X_t \\
   dY_t = \tilde{\rho}\left(X_t - Y_t\right) dB^{X}_t + \sqrt{1-\tilde{\rho}\left(X_t - Y_t\right)^2} dB^Y_t\\
\end{array}
\right.
\]
with $\tilde{\rho}$ a Lipschitz function such that $\underset{x \in \mathbb{R}}{\sup} |\tilde{\rho}\left(x\right) | < 1$, $\tilde{\rho}\left(x\right) = \rho_1$ if $x \leq \nu$ and $\tilde{\rho}\left(x\right) = \rho_2$ if $x \geq \eta $. This system of stochastic differential has a strong solution $\left(X,Y\right)$, see Proposition \ref{solutionsde}. This model seems to be equivalent to the multi-barrier model when the two barriers have close values and $\rho_2 = -\rho_1 = \rho$. The solution has the advantage to be Markovian.

\medskip
The multi-barrier correlation model is applied to the factorial model \eqref{factosde} in order to model jointly forward prices of electricity and forward prices of coal. Empirical results show that the model works well for products with a long delivery maturity (3 Month Ahead and 6 Month Ahead): the difference between the two products has an asymmetric distribution and the probability for the electricity product  to be higher than the coal one is high. However, it is not the case for products with a short delivery maturity, such as the spot. This can be explained by a difference of volatility too high between the electricity spot price and the coal spot price. Indeed, the electricity and coal volatilities of the long term factors that drives the prices of long maturity products are close to each other whereas they are very different for the short term factors. An other limitations of our model is that it is highly sensitive to initial conditions, that is the initial prices of electricity and coal products. We also estimate prices of European spread options in our model with Monte Carlo. Results are the same in the local correlation model.

\subsection{Structure of the paper} In Section \ref{twocorrelation}, we provide a first model to construct two Brownian motions with a two-state correlation structure based on the dependence between a Brownian motion and its reflection. We give a closed formula for the survival function of the difference between the two Brownian motions: the distribution of the difference is asymmetric and can take higher values than in the constant correlation case. In Section \ref{multibarriercorrelationmodel}, we improve the model of Section \ref{twocorrelation} by allowing several reflections to construct a multi-barrier correlation model. We give results about the survival function between the two Brownian motions and show that it takes higher values than the one in the model of Section \ref{twocorrelation}. We also derive a local correlation model which gives the same results than the multi-barrier correlation model. Section \ref{multibarriercorrelationmodel} is our major contributions. In Section \ref{application}, we apply our results to the modeling of the forward prices of two commodities which are electricity and coal and to the pricing of spread options. Proofs are given in Section \ref{proofs}.

\section{A two-state correlation copula}
\label{twocorrelation}

In this section, we derive a copula based on the Brownian motion and its reflection according to a barrier. As seen in \citep{deschatre16}, this copula contains two states depending on the value of the difference between the two Brownian motions: one of comonotonicity and one of countermonotonicity, that is correlation equal to 1 and -1. This copula maximizes $\mathbb{P}\left(B_t^1-B_t^2 \geq \eta \right)$ when the barrier is equal to $\frac{\eta}{2}$, see \cite[Proposition 3]{deschatre16}. However, the dependence between the two Brownian motions when it is modeled by these copulae is degenerated in the sense that the difference between the two Brownian motions becomes constant in an infinite horizon. In this section, we construct a copula which does not present this degeneracy but which allows higher values for $\mathbb{P}\left(B_t^1-B_t^2 \geq \eta \right)$ than in the Gaussian copula case. The idea is to relax the correlation: instead of having states of correlation with correlations equals to 1 and -1, we have states of correlation with correlations equals to $\rho$ and $-\rho$, $| \rho | < 1$. 

\subsection{Model} \label{modeltwocorrelation}

Let us consider a filtered probability space ($\Omega$, $\mathcal{F}$, $\left(\mathcal{F}_t\right)_{t \geq 0}$, $\mathbb{P}$) with $\left(\mathcal{F}_t\right)_{t \geq 0}$ satisfying the usual hypothesis (right continuity and completion) and $B^1 = \left(B^1_t\right)_{t \geq 0}$ a Brownian motion adapted to $\left(\mathcal{F}_t\right)_{t \geq 0}$. We denote by $\tilde{B}^{h}$ the Brownian motion reflection of $B$ on $x = h$ with $h \in \mathbb{R}$, i.e. $\tilde{B}^{h}_t =  - B^1_t + 2 (B^1_t - B^1_{\tau^h}){\bf1}_{t \geq \tau^h}$ with $\tau^h =  \inf \{t  \geq 0 : B^1_t = h \}$. Thus, $\tilde{B}^{k}$ is a $\mathcal{F}$ Brownian motion according to the reflection principle (see \cite[Theorem\ 3.1.1.2, p.\ 137]{jeanblanc09}). Let $\rho \in \left(0,1\right)$ and $Z$ a Brownian motion independent from $B^1$. We consider the stochastic process $B^2 = \rho \tilde{B}^{h}_t + \sqrt{1-\rho^2}Z$, which is a Brownian motion by L\'evy characterisation. 

\subsection{The copula} Let us recall that a function $C: \left[0,1\right]^2 \mapsto \left[0,1\right]$ is a copula if:
\begin{enumerate}
\item[(i)] $C$ is 2-increasing, i.e. $C\left(u_2,v_2\right) - C\left(u_1,v_2\right) + C\left(u_1,v_1\right) - C\left(u_2,v_1\right) \geq 0 \text{ for }  u_2 \geq u_1, v_2 \geq v_1$ and $u_1, u_2, v_1, v_2 \in \left[0,1\right]$,
\item[(ii)]  $C\left(u,0\right) = C\left(0,v\right) = 0$, $u, v \in \left[0,1\right]$,
\item[(iii)]  $C\left(u,1\right) = u, C\left(1,u\right) = u$,  $u \in \left[0,1\right]$.
\end{enumerate}
According to Skar's theorem \citep{sklar59}, if $X$ and $Y$ are two random variables with continuous distribution function $F^X$ and $F^Y$, there exists an unique copula $C$ such that $\mathbb{P}\left(X \leq x, Y \leq y\right) = C\left(F^X\left(x\right), F^Y\left(y\right)\right)$. We will call $C$ the copula of $\left(X,Y\right)$. 
\smallskip

In the following, we will denote by $\Phi$ the cumulative distribution function of a standard normal random variable and by $\Phi_{\rho}$ the cumulative distribution function of a bivariate gaussian vector of two standard normal random variables correlated with correlation $\rho$, $\rho \in \left(-1,1\right)$.

\medskip
Proposition \ref{nondegenerated} gives the copula between $B^1$ and $B^2$. 

\begin{proposition}[Proposition 3 of \citep{deschatre16}]  \label{nondegenerated}
Let $h > 0$, $t > 0$ and $\rho \in \left(0,1\right)$. The copula 
\[C_t(u,v) = \left\lbrace
\begin{array}{ccc}
\Phi_{\rho}\Bigl(\Phi^{-1}\left(u\right), \Phi^{-1}\left(v\right)+\frac{2\rho h}{\sqrt{t}}\Bigr) + v - \Phi\Bigl(\Phi^{-1}\left(v\right)+\frac{2\rho h}{\sqrt{t}}\Bigr) & \mbox{if} & \hspace{-0.5em} u  \geq \Phi\Bigl(\frac{h}{\sqrt{t}}\Bigr) \\
 \Phi_{-\rho}\Bigl(\Phi^{-1}\left(u\right),\Phi^{-1}\left(v\right)\Bigr) + \Phi_{\rho}\Bigl(\Phi^{-1}\left(u\right) - \frac{2h}{\sqrt{t}}, \Phi^{-1}\left(1-v\right)-\frac{2\rho h}{\sqrt{t}}\Bigr) +\\
   \Phi_{\rho}\Bigl(\Phi^{-1}\left(u\right)-\frac{2h}{\sqrt{t}}, \Phi^{-1}\left(v\right)\Bigr) -\Phi\Bigl(\Phi^{-1}\left(u\right)-\frac{2h}{\sqrt{t}}\Bigr)& \mbox{if} &\hspace{-0.5em} u  < \Phi\Bigl(\frac{h}{\sqrt{t}}\Bigr),
\end{array}\right.
\]
is the copula between $B^1_t$ and $B^2_t$ at time $t$ which are defined in the model of Section \ref{modeltwocorrelation}.
\end{proposition}

This copula is clearly asymmetric in the sense that $C_t\left(u,v\right) \neq C_t\left(v,u\right)$ for $u, v \in \left[0,1\right], \; t > 0$ which is a necessary condition if we want to have for $x > 0$, $\mathbb{P}\left(B^1_t - B^2_t \geq x\right) \geq \frac{1}{2}$, see \cite[Proposition 2]{deschatre16}. The copula contains two states of correlation: one of positive dependence ($\rho > 0$) and one of negative dependence ($\rho < 0$). Figure \ref{RBC} gives the copula of Proposition \ref{nondegenerated} with $\rho = 0.95$ and in the degenerated case $\rho = 1$, $h = 2$ and $t = 1$.

\begin{figure}[h!]
    \centering
    \begin{subfigure}[b]{0.32\textwidth}
        \centering
        \includegraphics[width=\textwidth]{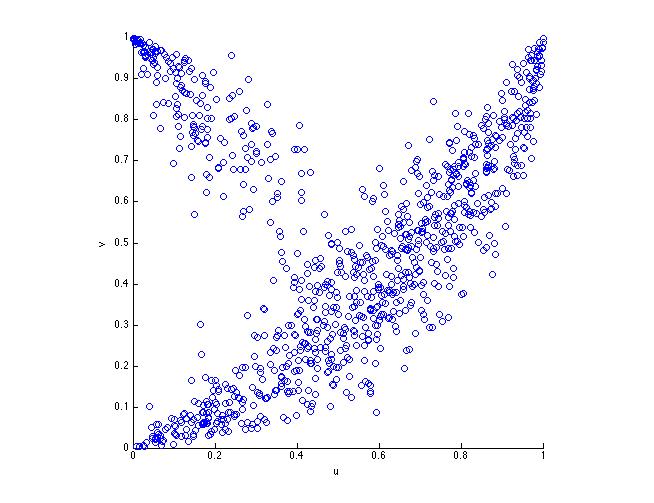}
        \caption{\label{RBCa}\it $\rho = 0.95$.}
    \end{subfigure}
    \begin{subfigure}[b]{0.32\textwidth}
        \centering
        \includegraphics[width=\textwidth]{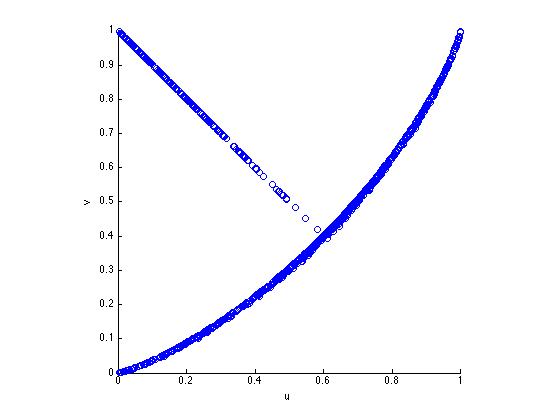}
        \caption{\label{RBCb} \it $\rho = 1$.}
        \end{subfigure}
     \caption{\label{RBC}\it Copula between a Brownian motion and Brownian motion correlated to the refection of the first one with a correlation $\rho = 0.95$ and in the degenerated case $\rho = 1$  at time $t = 1$ and a barrier $h=2$, which is the copula of Proposition \ref{nondegenerated}.}
\end{figure}

\subsection{Distribution of the difference between the two Brownian motions}

Proposition \ref{survival2} gives the survival function of $B^1_t-B^2_t$. 
\begin{proposition}  \label{survival2} Let $t > 0$, $h > 0$, $\rho \in \left(0,1\right)$ and $x \in \mathbb{R}$. Let $B^1$ and $B^2$ the two Brownian motions defined in the model of Section \ref{modeltwocorrelation}. We have: 
\[\mathbb{P}\left(B^1_t - B^2_t \geq x \right) = \Phi\Bigl(\frac{-x+2\rho h}{\sqrt{2\left(1-\rho\right)t}}\Bigr) \Phi\Bigl(\frac{x-2h\left(1+\rho\right)}{\sqrt{2\left(1+\rho\right)t}}\Bigr) + \Phi\Bigl(\frac{2h-x}{\sqrt{2\left(1-\rho\right)t}}\Bigr) \Phi\Bigl(\frac{-x}{\sqrt{2\left(1+\rho\right)t}}\Bigr).\]
\end{proposition}

Figure \ref{spreadonebarrier} represents the survival function of $B^1_t-B^2_t$ at time $t = 1$ and $t = 20$ with $h = 0.25$ and $\rho = 0.9$. The value of this function is close to 0.7 when $x$ is close to 0 at time $t = 1$. However, when $t = 20$, it becomes close to $\frac{1}{2}$ and the asymmetry disappears.
\begin{figure}[h!]
    \centering
    \begin{subfigure}[b]{0.49\textwidth}
        \centering
        \includegraphics[width=\textwidth]{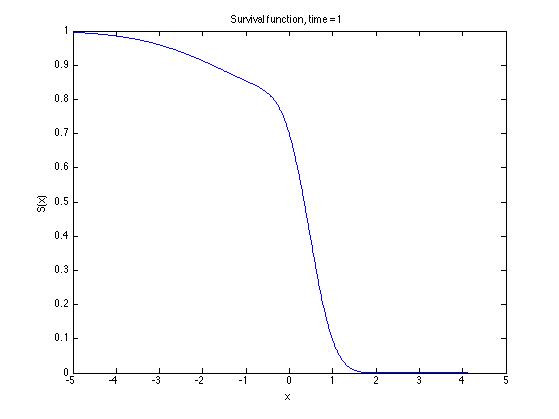}
        \caption{\label{spreadonebarriertime=1} \it $t = 1$.}
    \end{subfigure}
        \begin{subfigure}[b]{0.49\textwidth}
        \centering
        \includegraphics[width=\textwidth]{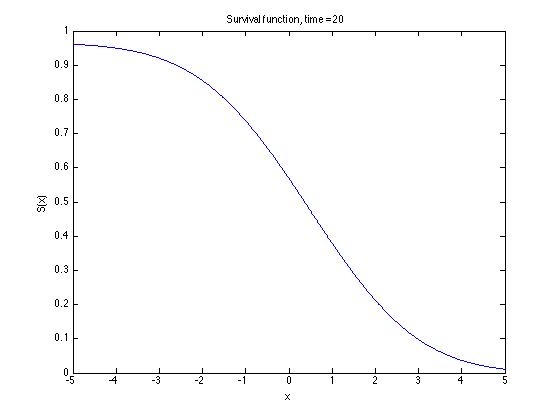}
        \caption{\label{spreadonebarriertime=20}\it $t = 20$.}
    \end{subfigure}
\caption{\label{spreadonebarrier} \it Survival function of $B_t^1 - B_t^2$ in the model of Section \ref{modeltwocorrelation} at time $t = 1$ and $t = 20$ with parameters $h = 0.25$ and $\rho = 0.9$.}
\end{figure}

This model allows us to have higher values than in the Gaussian copula case for $\mathbb{P}\left(B^1_t - B^2_t \geq z\right)$ when $z$ is close to 0. However, it presents some limitations in terms of modeling:
\begin{enumerate}
\item[(i)] $\left(B^1, \tilde{B}^h, B^2\right)$ is Markovian but the couple $\left(B,B^2\right)$ is not.
\item[(ii)]  The asymmetry disappears in the distribution of $B^1_t - B^2_t$ when $t$ becomes large.  
\item[(iii)] Let us consider the probability $\mathbb{P}\left(B^1_t - B^2_t \geq z \mid \mathcal{G}_s\right)$ with $\mathcal{G}$ the filtration generated by $\left(B^1, \tilde{B}^h, B^2\right)$. Let us suppose that the barrier has already been crossed at time $s$, i.e. $\tilde{B}^h_s = B^1_s-2h$. Thus, the correlation between $B^1$ and $B^2$ at times $t \geq s$ is equal to $\rho$ and does not change. We are in the same case than in the Gaussian copula case after time $s$, and then we do not optimize $\mathbb{P}\left(B^1_t - B^2_t \geq z \mid \mathcal{G}_s\right)$.
\end{enumerate}

\section{Multi-barrier correlation model}
\label{multibarriercorrelationmodel}

In this section, we improve the model of Section \ref{twocorrelation} by allowing several reflections. In the model of Section \ref{twocorrelation}, once the reflection has happened, the two Brownian motions stay correlated with correlation $\rho$ even if the difference between the two becomes low. We want to have two Brownian motions $X$ and $Y$ with the following correlation structure: if the value of $X-Y$ is under a certain level that we denote by $\nu$, $X$ and $Y$ have a negative correlation $-\rho$ and if it is over an other level denoted by $\eta$, their correlation is positive and equal to $\rho$. One way to obtain this structure is to start with two Brownian motions having a negative correlation. When the difference between them reaches the barrier $\eta$, $Y$ reflects and the correlation becomes positive. If the correlation is positive (resp. negative) and $X-Y$ reaches $\nu$ (resp. $\eta$), $Y$ reflects and the correlation becomes negative (resp. positive). The number of reflection that can happen is a parameter of our model denoted by $n$. $Y$ is then correlated to a reflection of $X$ reflecting each time the difference between the two reaches one of the two barriers. Figure \ref{multibarrierillustration} gives an illustration of our model. In Section \ref{localcorrelationmodel}, we develop a local correlation model based on the same principle. The local correlation model seems to be equivalent to the multi-barrier correlation model when the two barriers are close. Furthermore, in the local correlation model, the couple $\left(X,Y\right)$ is Markovian.
\begin{figure}[h!]
\centering
\includegraphics[width=0.7\linewidth]{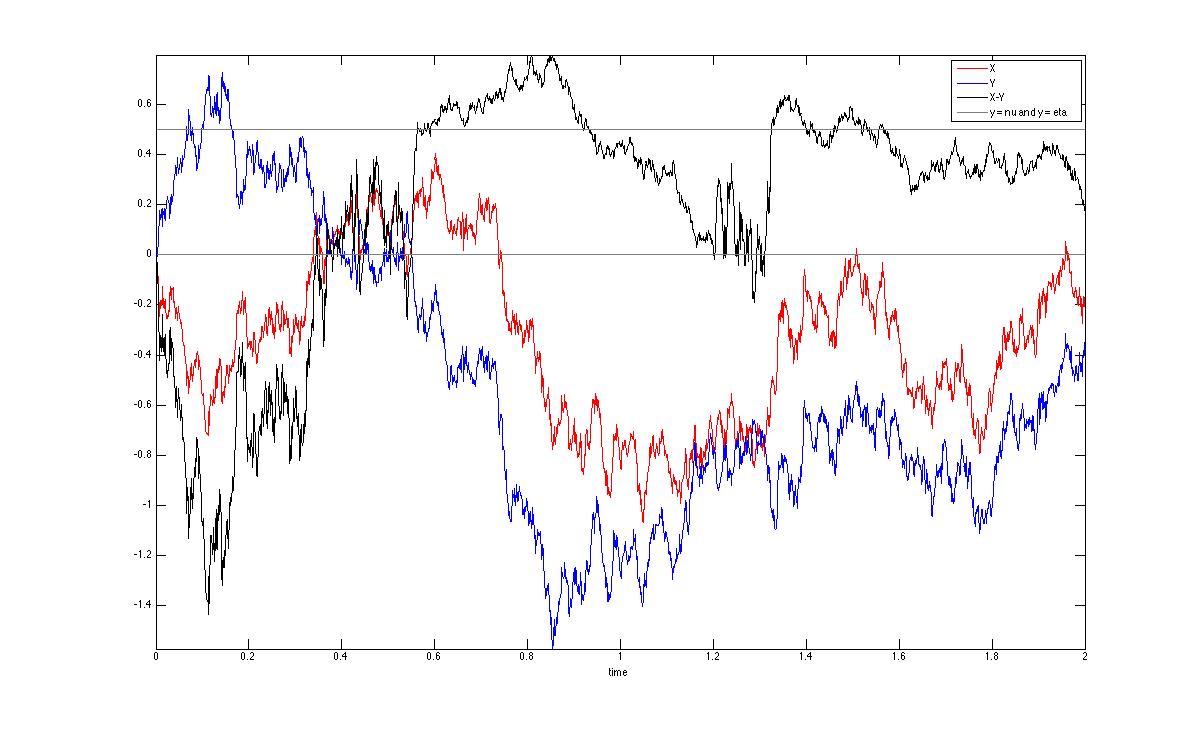}
\caption{ \label{multibarrierillustration} \it One trajectory of $X$, $Y$ and $X-Y$ in the multi-barrier correlation model with $\nu = 0$, $\eta = 0.5$, $\rho = 0.9$.}
\end{figure}

\subsection{Model}
\label{multimodel}
Let $B^X$ and $B^Y$ be two independent Brownian motions  defined on a common filtered probability space ($\Omega$, $\mathcal{F}$, $\left(\mathcal{F}_t\right)_{t \geq 0}$, $\mathbb{P}$) with $\left(\mathcal{F}_t\right)_{t \geq 0}$ satisfying the usual properties. We will denote indifferently $B^X$ by $X$.
\smallskip

\noindent Let $\eta > 0$, $\nu < \eta$ and $\rho \in \left[0,1\right]$.
\smallskip

\noindent Let $\alpha_k = \left\lbrace
\begin{array}{l}
0 \text{ if } k = 0\\
\eta \text{ if } k \text{ odd}\\
\nu \text{ if } k \text{ even, }k \neq 0 
\end{array}\right..
$
\smallskip

\noindent Let $\left(\tilde{B}^k\right)_{k \geq 0}$, $\left(Y^k\right)_{k \geq 0}$ and $\left(\tau_k\right)_{k \geq 0}$ be defined by 
\[\left\lbrace
\begin{array}{l}
\tau_0 = 0 \\
\tilde{B}^0 = -B^X\\
Y^{0}_t = \rho \tilde{B}^0 + \sqrt{1- \rho^2} B^Y
\end{array}\right.,
\]
\[\left\lbrace
\begin{array}{l}
\tau_{k} = \inf\{t \geq \tau_{k-1} : B^X_t - Y^{k-1}_t = \alpha_k\} \quad k \geq 1    \\
\tilde{B}^{k} = \mathcal{R}(\tilde{B}^{k-1}, \tau_k) \quad k \geq 1\\
Y^{k} = \rho\tilde{B}^k + \sqrt{1- \rho^2} B^Y \quad k \geq 1, 
\end{array}\right.
\]
where $\mathcal{R}(B,\tau)$ is the reflection Brownian motion of $B$ with the reflection happening at time $\tau$ and $\tau$ a stopping time, i.e. $\mathcal{R}(B,\tau)_t =  - B_t + 2 (B_t - B_{\tau}){\bf1}_{t \geq \tau}$.

Let $N_t = \sum_{n = 1}^{\infty} {\bf1}_{\tau_n \leq t}$ be the number of reflections that happened before time $t$ and $Y_t = Y_t^{N_t}$. $Y^N$ is well defined because $N_t < \infty$ almost surely according to Proposition \ref{model} (iii). Proposition \ref{model} gives results about the model.

\begin{proposition} \label{model}
\noindent (i) $\left(Y^k\right)_{k \geq 0}$ is a sequence of $\left(\mathcal{F}_t\right)_{t \geq 0}$ Brownian motions and $(\tau_k)_{k \geq 0}$ is a sequence of $\left(\mathcal{F}_t\right)_{t \geq 0}$ stopping times.

\smallskip
\noindent (ii)  For $t > 0$,
\begin{equation}
\label{valuespread}
X_t - Y^n_t =\left\lbrace  
\begin{array}{l}
\Bigl(1+\left(-1\right)^k\rho\Bigr)\left(B^X_t - B^X_{\tau_{k}}\right) - \sqrt{1-\rho^2}\left(B^Y_t - B^Y_{\tau_{k}}\right) + \alpha_k, \, \tau_{k}  \leq t \leq \tau_{k+1},\, 0 \leq k \leq n \\
\left(1+(-1)^n\rho\right)\left(B^X_t - B^X_{\tau_{n+1}}\right) - \sqrt{1-\rho^2}\left(B^Y_t - B^Y_{\tau_{n+1}}\right) + \alpha_{n+1}, \, \tau_{n+1}  \leq t 
\end{array}\right.\hspace{-0.5em}. 
\end{equation}

\smallskip
\noindent (iii) $N_t < \infty$ almost surely.

\smallskip
\noindent (iv) $Y$ is a Brownian motion.
\end{proposition}

Figure \ref{copulamultibarrier} is the empirical copula of $(X_t, Y^n_t)$ for different $n$ at time $t = 1$. The copula is asymmetric and we observe two states of correlation, as for the model of Section \ref{twocorrelation}.

\begin{figure}[h!]
    \centering
    \begin{subfigure}[b]{0.45\textwidth}
        \centering
        \includegraphics[width=\textwidth]{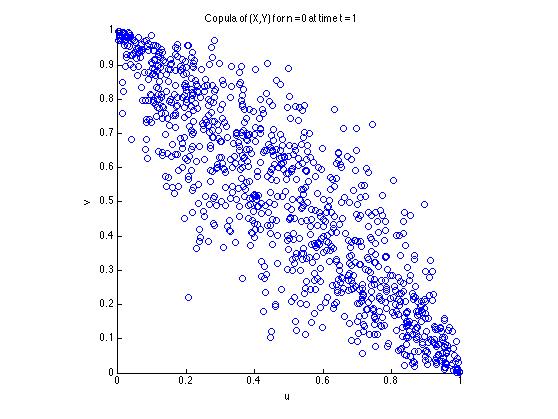}
        \caption{\it $n = 0$.}
    \end{subfigure}
    \begin{subfigure}[b]{0.45\textwidth}
        \centering
        \includegraphics[width=\textwidth]{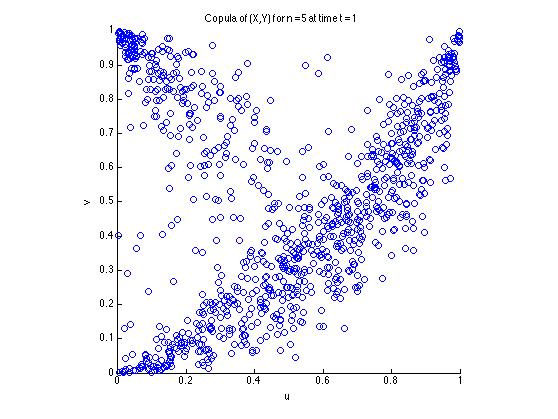}
        \caption{\it $n = 5$.}
    \end{subfigure}
        \begin{subfigure}[b]{0.45\textwidth}
        \centering
        \includegraphics[width=\textwidth]{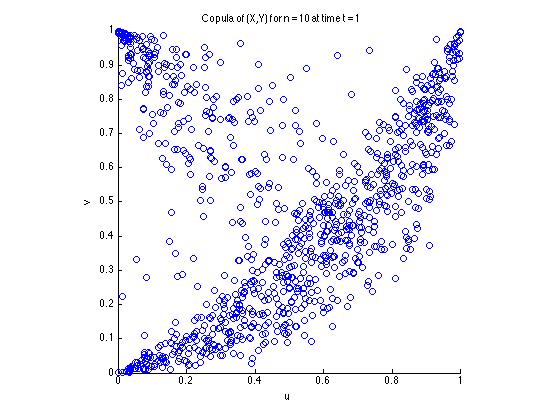}
        \caption{\it $n = 10$.}
    \end{subfigure}
        \begin{subfigure}[b]{0.45\textwidth}
        \centering
        \includegraphics[width=\textwidth]{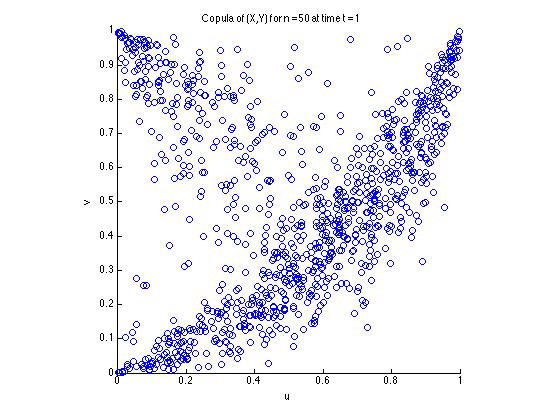}
        \caption{\it $n = 50$.}
    \end{subfigure}
    \caption{\label{copulamultibarrier} \it Empirical copula of $(X,Y^n)$ in the multi-barrier correlation model at time $t = 1$ with parameters $\nu$ = 0, $\eta = 0.5$ and $\rho = 0.9$ and a time step of 0.001 for different values of $n$ done with 1000 simulations.}
\end{figure}

\subsection{Results on the distribution of the difference between the two Brownian motions}

Proposition \ref{spreaddistrib} gives an analytic formula for the survival function for $X_t - Y_t^n$ and $X_t - Y_t$.

\medskip
\begin{proposition} \label{spreaddistrib}
Let $t > 0$ and $x \in \mathbb{R}$. Let $\left(p_n\left(t,x\right)\right)_{n \geq 0}$ the sequence defined by:
\begin{equation}
\label{spreadrecini}
p_0(t,x) = \Phi\Bigl(\frac{-x}{\sqrt{2\left(1+\rho\right)t}}\Bigr),
\end{equation}
\begin{equation}
\label{spreadrec}
p_{n}(t, x) = \left\lbrace
\begin{array}{ccc}
\Phi\Bigl(\frac{x-\alpha_{n+1}}{\sqrt{2\left(1+\left(-1\right)^n\rho\right)t}} - \frac{u_{n+1}}{\sqrt{t}}\Bigr) - \Phi\Bigl(\frac{x-\alpha_{n+1}}{\sqrt{2\left(1+\left(-1\right)^{n+1}\rho\right)t}} - \frac{u_{n+1}}{\sqrt{t}}\Bigr) & \mbox{if} & x < \alpha_{n+1}\\
\Phi\Bigl(\frac{x-\alpha_{n+1}}{\sqrt{2\left(1+\left(-1\right)^n\rho\right)t}} + \frac{u_{n+1}}{\sqrt{t}}\Bigr) - \Phi\Bigl(\frac{x-\alpha_{n+1}}{\sqrt{2\left(1+\left(-1\right)^{n+1}\rho\right)t}} + \frac{u_{n+1}}{\sqrt{t}}\Bigr) & \mbox{if} & x \geq \alpha_{n+1}
\end{array}\right.
\end{equation}
where $\left(u_n\right)_{n\geq0}$ is the sequence defined by:
\[\label{uk}
\left\lbrace
\begin{array}{l}
u_0 = 0\\
u_{n} = \frac{\eta}{\sqrt{2\left(1+\rho\right)}} + \frac{\left(\eta - \nu\right)}{\sqrt{2}}\Bigl(\frac{\lfloor \frac{n}{2} \rfloor}{\sqrt{1-\rho}} + \frac{\lfloor \frac{n-1}{2} \rfloor}{\sqrt{1+\rho}}\Bigr) \quad k \geq 1  \\
\end{array}\right.
\]
and $\lfloor . \rfloor$ is the floor function.

\medskip 
We have:
\[\mathbb{P}\left(X_t - Y^n_t \geq x \right) = \sum_{k=0}^{n} p_k\left(t,x\right)\]
and
\[\mathbb{P}\left(X_t - Y_t \geq x \right) = \sum_{k=0}^{\infty} p_k\left(t,x\right).\]
\end{proposition}

\begin{corollary} \label{convergencesurvival} Let $t > 0$. For $x \in \left[\nu, \eta\right]$, the sequence $\mathbb{P}\left(X_t - Y^n_t \geq x\right)$ is increasing with $n$ when $\rho > 0$.
\end{corollary}

\medskip
For $x \in \left[\nu, \eta\right]$, the survival function takes higher values than in the constant correlation case and than $\frac{1}{2}$. Furthermore, it is possible to increase the value of $\mathbb{P}\left(X_t-Y^n_t \geq x\right)$ by increasing the number of reflections with this model, which is why the case $n = \infty$ is considered.

\medskip
Results of Proposition \ref{convergencesurvival} are illustrated in Figure \ref{spreadmultibarriertime=1}. The case $n = 0$ corresponds to the Gaussian case. We can see that in $\left[ \nu,\eta \right]$, the survival function is increasing with $n$. In Figure \ref{spreadmultibarriertime=1}, the curves for $n = 5$, $n = 10$ and $n = 50$ are the same. At time $t = 1$, the probability to cross more than 5 barrier is very weak then the Brownian reflection reflects less than 5 times with a high probability. The convergence in $n$ at small time is fast. In Figure \ref{spreadmultibarriertime=20}, we can observe the difference between the cases $n = 5$, $n = 10$ and $n = 50$ at time $t = 20$. The survival function continues to grow. The survival function does not present the problem of the one the model of Section \ref{twocorrelation}: its value stays high when $t = 20$ which is caused by the several reflections.

\begin{figure}[h!]
    \centering
    \begin{subfigure}[b]{0.49\textwidth}
        \centering
        \includegraphics[width=\textwidth]{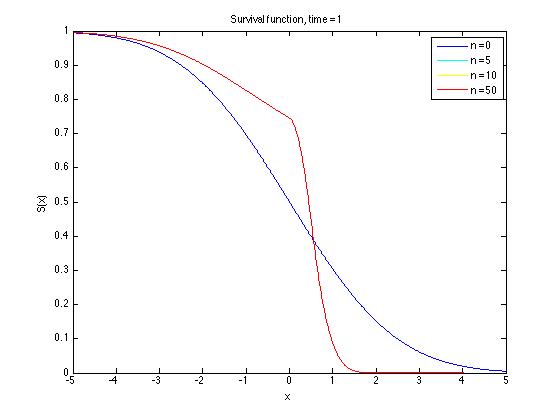}
        \caption{\label{spreadmultibarriertime=1} \it $t = 1$.}
    \end{subfigure}
        \begin{subfigure}[b]{0.49\textwidth}
        \centering
        \includegraphics[width=\textwidth]{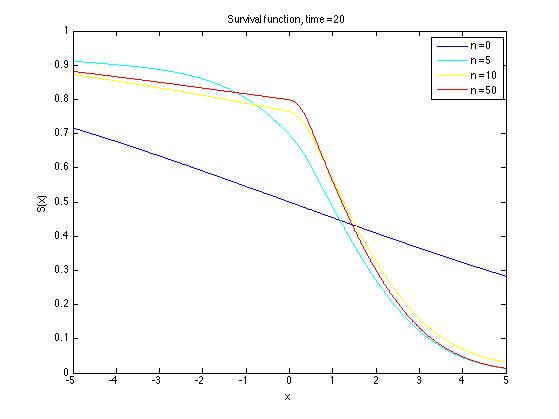}
        \caption{\label{spreadmultibarriertime=20}\it $t = 20$.}
    \end{subfigure}
\caption{\label{spreadmultibarriertime} \it Survival function of $X - Y^n$ in the multi-barrier correlation model at time $t$ with parameters $\nu$ = 0, $\eta = 0.5$ and $\rho = 0.9$ for different values of $n$.}
\end{figure}

\medskip
The results are confirmed with Figure \ref{simuspreadmultibarrier}. The higher the number of reflections is, more $X-Y^n$ is concentrated in the region $\left[ \nu,\eta \right]$. However, in the positive part of the plan, $X-Y^n$ take lower values than in the Gaussian case $n = 0$. One explanation comes from the martingality of $X-Y^n$. As $X-Y^n$ is a martingal, we have $\mathbb{E}\left(X_t - Y^n_t\right) = \mathbb{E}\left(X_0 - Y^n_0\right) = 0$.  Furthermore, $\mathbb{P}\left(X_t - Y^n_t \geq 0\right) > \frac{1}{2}$ and is higher than in the case of a constant correlation between the two Brownian motions. The probability mass in the positive part of the real line increases, but the expectation on all the real line stay the same: values taken by the random variables become lower in the positive part of the real line and becomes higher in the negative.We also remark that the symmetry present in the case $n = 0$ disappears when $n$ is higher.

\begin{figure}[h!]
    \centering
    \begin{subfigure}[b]{0.45\textwidth}
        \centering
        \includegraphics[width=\textwidth]{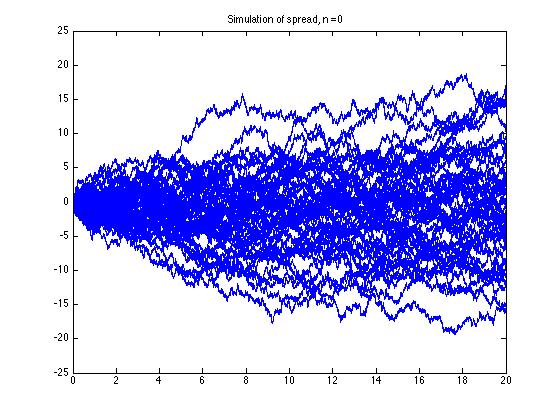}
        \caption{\it $n = 0$.}
    \end{subfigure}
    \begin{subfigure}[b]{0.45\textwidth}
        \centering
        \includegraphics[width=\textwidth]{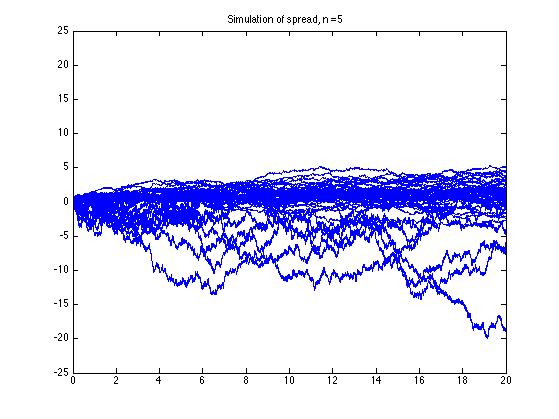}
        \caption{\it $n = 5$.}
    \end{subfigure}
        \begin{subfigure}[b]{0.45\textwidth}
        \centering
        \includegraphics[width=\textwidth]{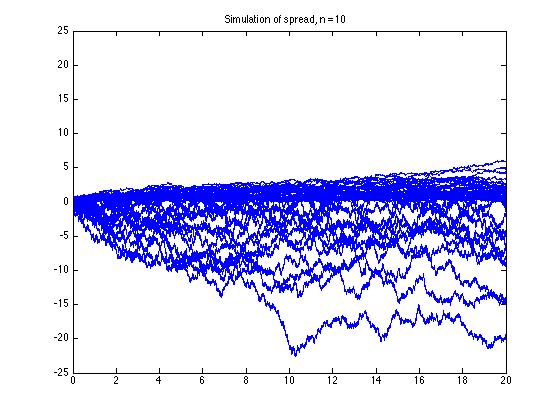}
        \caption{\it $n = 10$.}
    \end{subfigure}
        \begin{subfigure}[b]{0.45\textwidth}
        \centering
        \includegraphics[width=\textwidth]{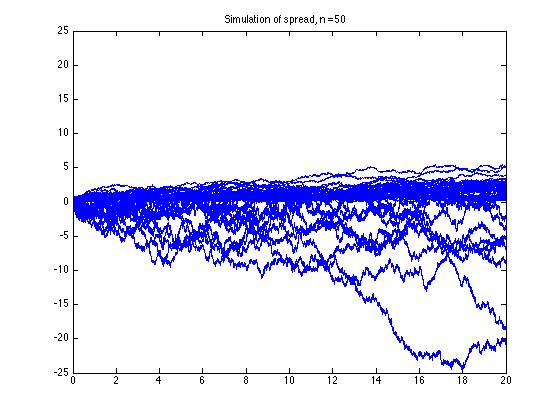}
        \caption{\it $n = 50$.}
    \end{subfigure}
    \caption{\label{simuspreadmultibarrier} \it 50 simulations of $X-Y^n$ in the multi-barrier correlation model between time 0 and 20 with parameters $\nu$ = 0, $\eta = 0.5$ and $\rho = 0.9$ and a time step of 0.001 for different values of $n$.}
\end{figure}

\subsection{A local correlation model}
\label{localcorrelationmodel}

As in Section \ref{multimodel}, we develop a model based on a two-states structure of correlation. However, we use a totally different approach where the reflection of the Brownian motion does not appear. Our model is a local correlation model and the correlation depends on the value of the difference between the two Brownian motions. The local correlation function presents two states of correlation: one of negative correlation if the difference of the two Brownian motion is under a certain barrier, one of positive correlation if the difference if over an other barrier and between the two barriers the function is chosen with sufficient regularity.

\medskip
Let $B^X$ and $B^Y$ be two independent Brownian motions defined on a filtered probability space ($\Omega$, $\mathcal{F}$, $\left(\mathcal{F}_t\right)_{t \geq 0}$, $\mathbb{P}$).
\smallskip

\noindent Let $\eta, \nu, \rho_{\min}$ and $\rho_{\max}$ be real numbers with $\eta > \nu$, $| \rho_{\min} | < 1$, $| \rho_{\max} | < 1$.
\smallskip

\noindent Let $\tilde{\rho}(x)$ be a function such that $\tilde{\rho}\left(x\right) = \rho_{\min}$ for $x \leq \nu$, $\tilde{\rho}\left(x\right) =\rho_{\max}$ for $x \geq \eta$ and $\underset{x \in \mathbb{R}}{\sup} |\tilde{\rho}\left(x\right) | < 1$. Let us assume that $\tilde{\rho}$ is Lipschitz.

\medskip
Let us consider the following system of stochastic differential equations:
\begin{equation}
\label{sdebrownianlocal}
\left \{
\begin{array}{l}
   dX_t = dB^X_t \\
   dY_t = \tilde{\rho}\left(X_t - Y_t\right) dB^{X}_t + \sqrt{1-\tilde{\rho}\left(X_t - Y_t\right)^2} dB^Y_t\\
\end{array}
\right.
\end{equation}
with $X_0 = 0$ and $Y_0 = 0$.
\medskip

Proposition \ref{solutionsde} gives results about the solution of \eqref{sdebrownianlocal}.

\begin{proposition} \label{solutionsde} The system of stochastic differential equations \eqref{sdebrownianlocal} has an unique strong solution $\left(X,Y\right)$ with $X$ and $Y$ two Brownian motions. Furthermore, $\left(X,Y\right)$ is Markovian.
\end{proposition}

Contrary to the multi-barrier correlation model, the local correlation model has the advantage to give a Markovian solution, which has some importance in practice. However, less analytical results are available for this model. In the following, we gives empirical results about it. Results are close to the ones of the multi-barrier correlation model. 

\medskip
As the local correlation function is asymmetric, i.e. $\rho\left(x,y\right) \neq \rho\left(x,y\right), \; x, y \in \mathbb{R}$, the copula of the solution of \eqref{sdebrownianlocal} is expected to be asymmetric. Figure \ref{copulalocal} represents the copula of $(X_t, Y_t)$ at time $t=1$. It is similar to the one of the multi-barrier correlation model.

\begin{figure}[h!]
\centering
\includegraphics[height=0.3\textheight]{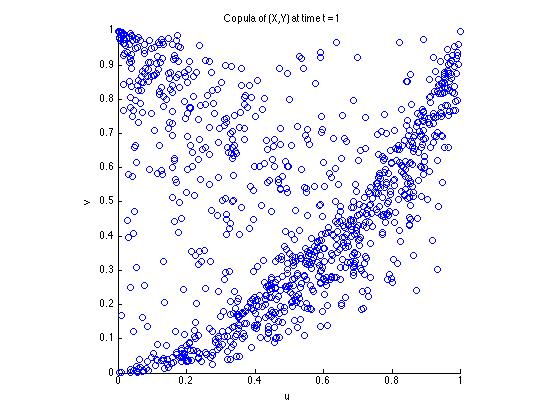}
\caption{\label{copulalocal} \it Empirical copula of $(X_t,Y_t)$ in the local correlation model at time $t = 1$ with parameters $\nu$ = 0, $\eta = 0.5$, $\rho_1 = - 0.9$ and $\rho_2 = 0.9$ and a time step of 0.001 with 1000 simulations.}
\end{figure}

\medskip
Figure \ref{spreadlocaltime=1} represents the survival function of the $X_t-Y_t$ in the local correlation model at time $t = 1$ and $t = 20$ with parameters $\nu$ = 0, $\eta = 0.5$, $\rho_{\min} = - 0.9$ and $\rho_{\max} = 0.9$. The local correlation function is chosen linear between $\nu$ and $\eta$. As for the multi-barrier correlation model, the distribution of $X_t - Y_t$ is asymmetric. The survival function seems equivalent to the one of the multi-barrier correlation model. Between $\nu$ and $\eta$, the survival function is over $\frac{1}{2}$ (Gaussian copula case). The survival function increases at the right of $\nu$ between time $t = 1$ and $t = 20$. 

\begin{figure}[h!]
    \centering
    \begin{subfigure}[b]{0.49\textwidth}
        \centering
        \includegraphics[width=\textwidth]{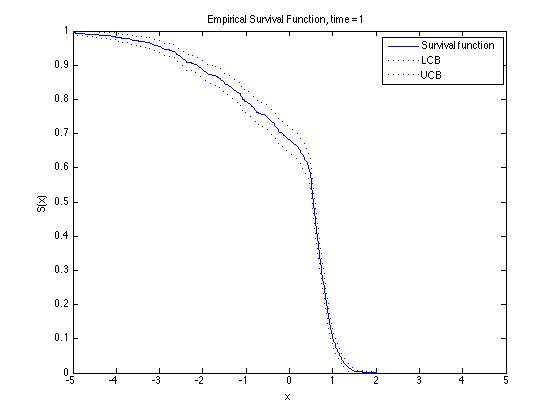}
        \caption{\it $t = 1$.}
    \end{subfigure}
        \begin{subfigure}[b]{0.49\textwidth}
        \centering
        \includegraphics[width=\textwidth]{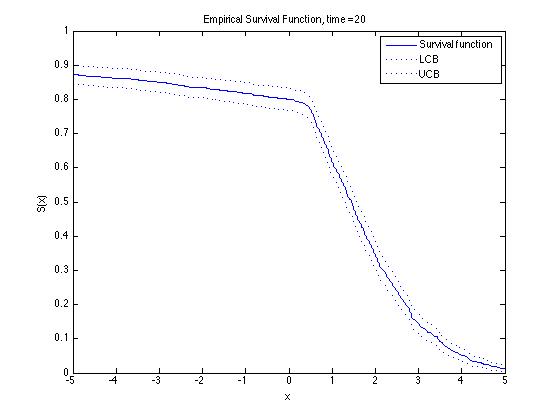}
        \caption{\it $t = 20$.}
    \end{subfigure}
\caption{\label{spreadlocaltime=1} \it Empirical survival function of $X_t - Y_t$ in the local correlation model at time $t$ with parameters $\nu$ = 0, $\eta = 0.5$, $\rho_{\min} = - 0.9$ and $\rho_{\max} = 0.9$ with interval confidence bounds at $99\%$ and estimated with 1000 simulations and a step time of 0.001.}
\end{figure}

\medskip
Figure \ref{simulocal} represents 50 simulations of $X-Y$ in the correlation local model with parameters $\nu$ = 0, $\eta = 0.5$, $\rho_{\min} = - 0.9$ and $\rho_{\max} = 0.9$.  As for the multi-barrier correlation model, the trajectories are concentrated in the positive part of the plan. 

\begin{figure}[h!]
\centering
        \includegraphics[width=0.5\textwidth]{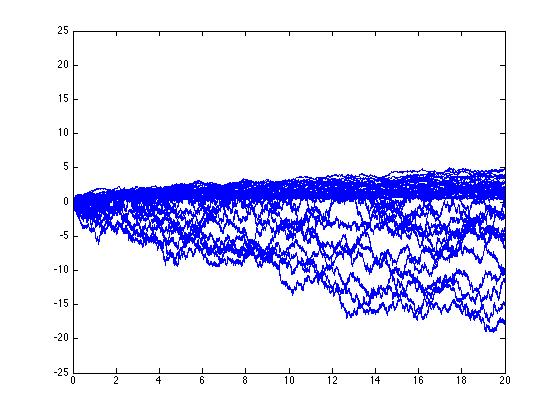}
 \caption{\label{simulocal} \it 50 simulations of $X-Y$ in the correlation local model with parameters $\nu$ = 0, $\eta = 0.5$, $\rho_{\min} = - 0.9$ and $\rho_{\max} = 0.9$ between time $t = 0$ and $t = 20$  and a time step of 0.001.}
\end{figure}

\section{An application for joint modeling of commodity prices on energy market}
\label{application}

In this section, we use the multi-barrier correlation model for the joint modeling of the forward prices of two commodities, electricity and coal. Coal is a fuel used to produce electricity which implies an asymmetry in the distribution of the difference between the two prices ; it is more likely that price of coal is lower than price of electricity (in the same unit). Modeling the dependence with a Gaussian copula is then not adapted. An advantage of our model is that it contains asymmetry in the distribution of the difference between the two prices. Furthermore, it allows not to change the marginal models.

\subsection{Model}
\medskip

Let us consider a two-factor model for both electricity and coal. For more information on the two-factor model, we refer to the study of Benth and Koekebakker \citep{benth08}.

\medskip
Let $f^{E}\left(t,T\right)$ (resp. $f^{C}\left(t,T\right)$) the forward price of the electricity (resp. coal) at time $t$ with maturity $T$, that is of the delivery of electricity (resp. coal) at maturity $T$ during one day. Stochastic differential equation \eqref{sdecommo} gives dynamic of these products.

\begin{equation}
\label{sdecommo}
\left \{
\begin{array}{c @{=} c}
 df^E\left(t,T\right) & f^E\left(t,T\right)\left(\sigma^E_s e^{-\alpha^E_s(T-t)} dB^{E,s}_{t} + \sigma^E_l dB^{E,l}_{t}\right)\\
 df^C\left(t,T\right) & f^C\left(t,T\right)\left(\sigma^C_s e^{-\alpha^C_s(T-t)} dB^{C,s}_{t} + \sigma^C_l dB^{C,l}_{t}\right)   
\end{array}
\right.
\end{equation}
where $B^{E,s}$, $B^{E,l}$, $B^{C,s}$, $B^{C,l}$ are standard Brownian motions defined on a common probability space $\left(\Omega, \mathcal{F}, \mathbb{P}\right)$.

\medskip 
In the dynamic of each commodity, there is one factor corresponding to the short term factor with a volatility $\sigma^i_s e^{-\alpha^i_s(T-t)}, i = E, C$ . This short term factor is used to model the Samuelson effect \citep{samuelson65}, that is the decrease of volatility with time to maturity. The other factor is the long term factor with a constant volatility $\sigma^i_{l}, i = E, C$. 

\medskip
Products traded on the market have a delivery period, except for the spot. We denote by $f^i\left(t,T,\theta\right), i =E,C$ the price of the product at time $t$ that delivers $i$ at time $T$ during a period $\theta$. By absence of arbitrage opportunities, we have 
\[f^i\left(t,T,\theta\right) = \frac{1}{\theta}\int_{T}^{T + \theta} f^i\left(t,u\right)du.\]
In the following, we will only consider $n$ Month Ahead ($n$MAH), $n \geq 1$, which are products with a delivery period of one month and a delivery date which is the $1^{st}$ of the $n^{th}$ following month from today.

\medskip
Equation \eqref{solsdecommo} gives the solutions of \eqref{sdecommo}.
\begin{equation}
\label{solsdecommo}
\left \{
\begin{array}{c @{=} c}
f^E\left(t,T\right) & f^E\left(0,T\right)e^{\int_0^t \sigma^E_s e^{-\alpha^E_s\left(T-u\right)} dB^{E,s}_{u} - \frac{1}{2}\int_0^t \left(\sigma^E_s\right)^2 e^{-2\alpha^E_s\left(T-u\right)} du + \sigma^E_l B^{E,l}_{t} - \frac{1}{2}\left(\sigma^E_l\right)^2t}\\
f^C\left(t,T\right) & f^C\left(0,T\right)e^{\int_0^t \sigma^C_s e^{-\alpha^C_s\left(T-u\right)} dB^{C,s}_{u} - \frac{1}{2}\int_0^t \left(\sigma^C_s\right)^2 e^{-2\alpha^C_s\left(T-u\right)} du + \sigma^C_l B^{C,l}_{t} - \frac{1}{2}\left(\sigma^C_l\right)^2 t}\\\end{array}
\right.
\end{equation}

The spot price of electricity is given by $S^E_t = f^E(t,t)$ and the one of coal by $S^C_t = f^C(t,t)$. Then we have 
\begin{equation}
\left \{
\begin{array}{c @{=} c}
S^E_t & f^E\left(0,t\right)e^{\int_0^t \sigma^E_s e^{-\alpha^E_s\left(t-u\right)} dB^{E,s}_{u} - \frac{1}{2}\int_0^t \left(\sigma^E_s\right)^2 e^{-2\alpha^E_s\left(t-u\right)} du + \sigma^E_l B^{E,l}_{t} - \frac{1}{2}\left(\sigma^E_l\right)^2t}\\
S^C_t & f^C\left(0,t\right)e^{\int_0^t \sigma^C_s e^{-\alpha^C_s\left(t-s\right)} dB^{C,s}_{u} - \frac{1}{2}\int_0^t \left(\sigma^C_s\right)^2 e^{-2\alpha^C_s\left(t-u\right)} du + \sigma^C_l B^{C,l}_{t} - \frac{1}{2}\left(\sigma^C_l\right)^2 t}\\\end{array}
\right.
\end{equation}

We model the dependence as follow:
\begin{itemize}
\item $B^{E,s}$ and $B^{E,l}$ are independent,
\item $B^{C,s}$ and $B^{C,l}$ are independent,
\item $B^{E,s}$ and $B^{C,s}$ are independent,
\item $B^{E,l}$ and $B^{C,l}$ are constructed following the multi-barrier correlation model defined in Section \ref{multibarriercorrelationmodel}. 
\end{itemize}

Usually, a constant correlation matrix is used to model the dependence between the 4 Brownian motions. 

\subsection{Parameters}
\label{parameters}
\medskip

We consider the parameters of the marginal laws given in Table \ref{paramgazelec}. Units are taken according to the year. We use the forward prices of electricity and of coal during 2014 in France to estimate these parameters. The method used for estimation is the first one of \citep{feron15}. 

\begin{table}[h!]
\centering
\begin{tabular}{|c|c|c|}
\hline
Parameters & Electricity & Coal \\
\hline
$\sigma_l$ &10.2555$\%$ & 9.2602$\%$  \\
\hline
$\sigma_s$ & $97.2925\%$ & 11.2134$\%$ \\
\hline
$\alpha_s$ & 17.0363 & 2.07832 \\
\hline
\end{tabular}
\caption{\label{paramgazelec} \it Parameters of the two-factor model for electricity and coal.}
\end{table}

Parameters for the multi-barrier correlation model used to model the dependence between $B^{E,l}$ and $B^{C,l}$ are chosen arbitrarily ; we choose $\nu = 0$,  $\eta = 0.5$, $\rho = 0.9$, $n = \infty$.

\medskip
In the benchmark model where dependence between $B^{E,l}$ and $B^{C,l}$ is modeled by a constant correlation, the correlation is equal to 0.275. The other correlation are equals to 0.

\medskip
We assume that $f^E\left(0,T\right) - H f^C\left(0,T\right) = 0$ and $f^E\left(0,T\right) = 100$ for all $T$ (which does not represent the reality because we do not take into account the seasonality of the prices of electricity and coal). $H$ is a conversion factor between the unit of electricity prices and the unit of coal prices and is called the heat rate. 

\subsection{Numerical results}
We are interested in the difference between $f^E\left(t,T\right)$ and $H f^C\left(t,T\right)$. We only are interested in the multi-barrier correlation model ; results are the same for the local correlation model. 

\begin{figure}[h!]
    \centering
    \begin{subfigure}[b]{0.45\textwidth}
        \centering
        \includegraphics[width=\textwidth]{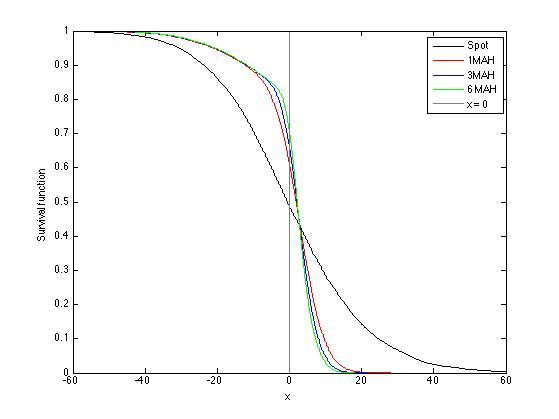}
        \caption{\it Multi-barrier correlation model.}
    \end{subfigure}
    \begin{subfigure}[b]{0.45\textwidth}
        \centering
        \includegraphics[width=\textwidth]{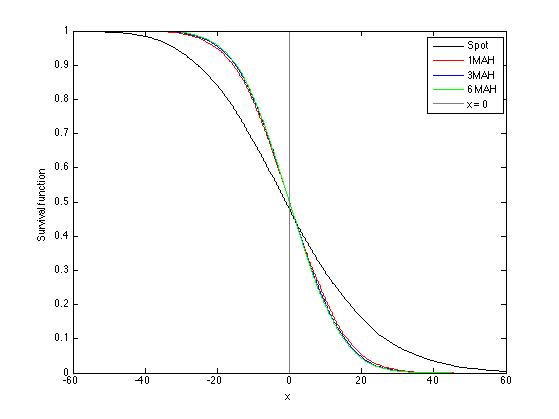}
        \caption{\it Benchmark model.}
    \end{subfigure}
    \caption{\label{differencedistribution} \it Empirical survival function of the difference between the price of electricity and the price of coal at time $t = 365$ days estimated with 10000 simulations with a time step of $\frac{1}{24}$ days for different products (Spot, 1MAH, 3MAH, 6MAH) in the multi-barrier correlation model and in the benchmark model.}
\end{figure}

\medskip
Figure \ref{differencedistribution} represents the survival function of the difference between spot, 1MAH, 3MAH, and 6MAH prices. In the multi-barrier correlation model, the probability for the difference between the two spot prices to be non negative is close to $50\%$, which is the same value than in the benchmark model. However, we have good results if we consider long term products as 1MAH, 3MAH and 6MAH: we have probabilities closed to $60\%$ for the 1MAH, and $70\%$ for the 3MAH and 6MAH in the multi-barrier correlation model whereas we have probabilities closed to $50\%$ in the benchmark model. The probability increases with the time to maturity. In the case of spot prices, the volatilities of the prices of the commodities is dominated by the short term factor, which we do not control ; in the other cases, these volatilities are small and the long term factor which we control dominates. This explains that we do not increase a lot the probability for the difference between the spot prices to be non negative. We also observed that in the multi-barrier correlation model, the survival function decreases faster than in the benchmark model and probability of being superior to 20 is closed to 0, which is not the case in the benchmark model. 

\begin{figure}[h!]
    \centering
    \begin{subfigure}[b]{0.45\textwidth}
        \centering
        \includegraphics[width=\textwidth]{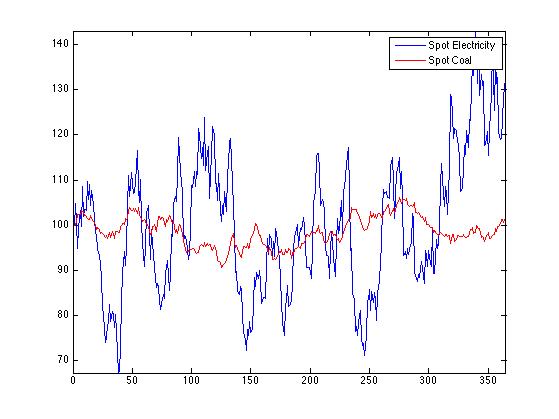}
        \caption{\it Spot prices of electricity and coal.}
    \end{subfigure}
    \begin{subfigure}[b]{0.45\textwidth}
        \centering
        \includegraphics[width=\textwidth]{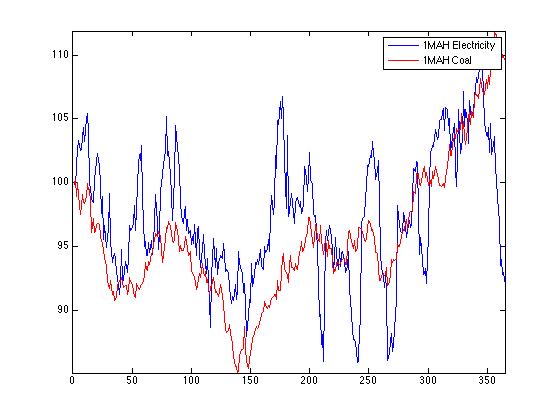}
        \caption{\it 1MAH prices of electricity and coal.}
    \end{subfigure}
        \begin{subfigure}[b]{0.45\textwidth}
        \centering
        \includegraphics[width=\textwidth]{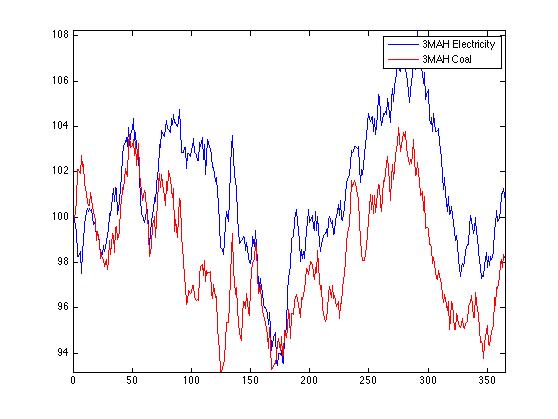}
        \caption{\it 3MAH prices of electricity and coal.}
    \end{subfigure}
        \begin{subfigure}[b]{0.45\textwidth}
        \centering
        \includegraphics[width=\textwidth]{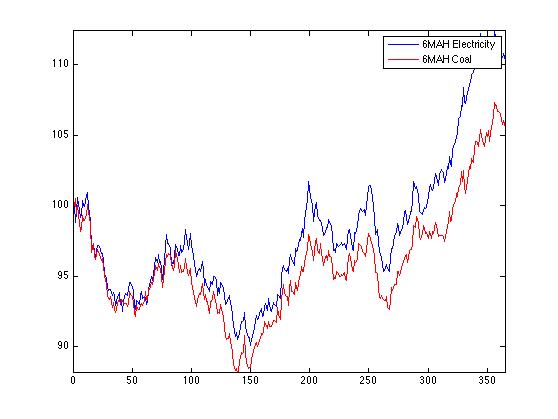}
        \caption{\it 6MAH prices of electricity and coal.}
    \end{subfigure}
    \caption{\label{trajectoryexample} \it One year trajectory of electricity and coal products in the multi-barrier correlation model with a time step of $\frac{1}{24}$ days.}
\end{figure}

\medskip
Figure \ref{trajectoryexample} represents one trajectory of the different products. In the case of the spot prices, since electricity has a high volatility, it is difficult to control the difference between the two processes. For the other products, as the short term volatility decreases, we see that there is a control between the two processes.

\begin{figure}[h!]
    \centering
    \begin{subfigure}[b]{0.45\textwidth}
        \centering
        \includegraphics[width=\textwidth]{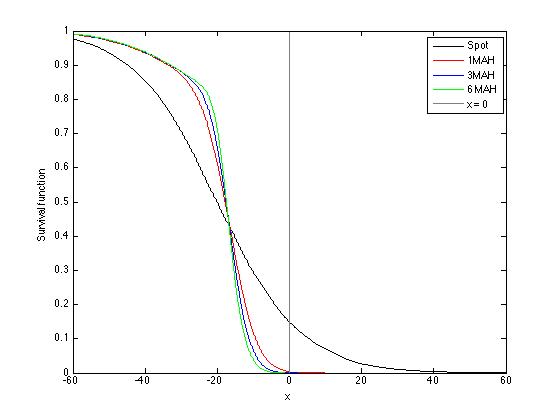}
        \caption{\it Multi-barrier correlation model.}
    \end{subfigure}
    \begin{subfigure}[b]{0.45\textwidth}
        \centering
        \includegraphics[width=\textwidth]{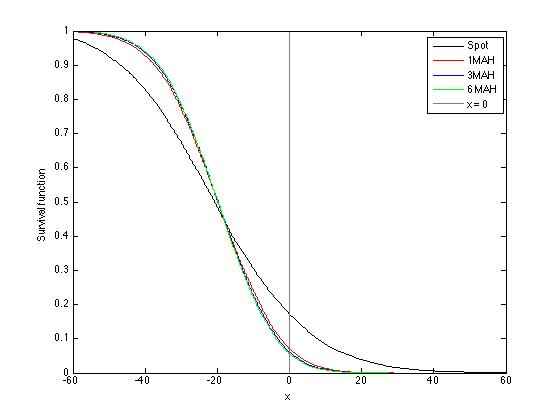}
        \caption{\it Benchmark model.}
    \end{subfigure}
    \caption{\label{differencedistribution20} \it Empirical survival function of the difference between the price of electricity and the price of coal at time $t = 335$ days estimated with 10000 simulations with a time step of $\frac{1}{24}$ days for different products (Spot, 1MAH, 3MAH, 6MAH) in the multi-barrier correlation model and in the benchmark model if the difference is equal to -20 at time $t = 0$.}
\end{figure}

\begin{remark}
Using a multi-barrier correlation model to model the dependence between $B^{E,s}$ and $B^{C,s}$ does not improve the results for the different survival functions. That is why we consider them independent. 
\end{remark}

\medskip
Results are sensitive to initial conditions. If we choose $f^E\left(0,T\right) = 100$ and  $H f^C\left(0,T\right) = 120$ for instance, $f^E\left(0,T\right)-H f^C\left(0,T\right) = -20$ and we will have a distribution that is concentrated around -20, because the difference between the price is a martingale. The probability to be greater than -20 is higher in the multi-barrier correlation model than in the benchmark model but the probability to be positive is lower than in the benchmark model: it is closed to 0 in the multi-barrier correlation model whereas it is closed to $10\%$ in the benchmark model. Figure $\ref{differencedistribution20}$ represents the survival function of the difference between prices of electricity and coal for different products with $\nu = 0$ and $\eta = 0.5$. As we choose a barrier near 0, the survival function will be maximized around -20.

\medskip
One way to improve the value of the survival function around 0 is to choose a higher $\eta$. The idea in our model is that we want $B^{E,l}$ to go over $B^{C,l} + \eta$, using correlation of -1 when the two prices are equals at time $t = 0$. We want for the price of the electricity to go over the price of coal, that happens when $f^E\left(t,T\right) = H f^C\left(t,T\right)$, i.e. when $\sigma^E_l B^{E,l}_t - \sigma^C_l B^{C,l}_t= \log\left(\frac{Hf^C\left(0,T\right)}{f^E\left(0,T\right)}\right)$ if we neglect the short term factors. We have $\sigma^E_l \approx \sigma^C_l \approx \sigma = 0.1$ year$^{-1}$. Then, we want $B^{E,l}_t - B^{C,l}_t \approx \frac{1}{\sigma}\log\left(\frac{Hf^C\left(0,T\right)}{f^E\left(0,T\right)}\right)$. In the case with the same initial conditions, the right hand side term is equal to $0$ and we choose a barrier of $\eta$. Heuristically, we then choose a barrier of $\eta^{'} = \eta + \frac{1}{\sigma}\log\left(\frac{Hf^C\left(0,T\right)}{f^E\left(0,T\right)}\right) \approx 170.5$ and $\nu = 170$. Figure \ref{differencedistribution20eta2} gives the survival function of the different products in the multi-barrier correlation model with barriers $\nu = 170$ and $\eta =170.5$. 

\begin{figure}[h!]
\centering
\includegraphics[height=0.3\textheight]{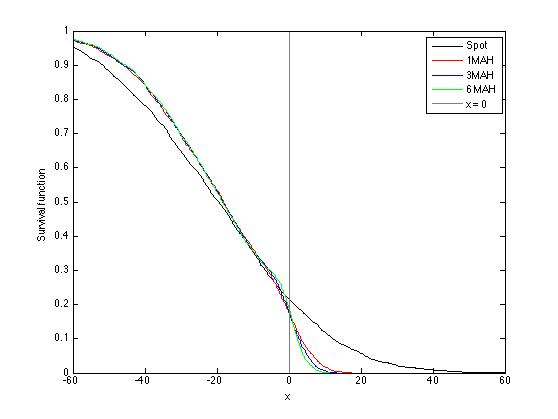}
  \caption{\label{differencedistribution20eta2} \it Empirical survival function of the difference between the price of electricity and the price of coal at time $t = 335$ days estimated with 10000 simulations with a time step of $\frac{1}{24}$ days for different products (Spot, 1MAH, 3MAH, 6MAH) in the multi-barrier correlation model if the difference is equal to -20 at time $t = 0$ and with barriers $\nu = 170$ and $\eta =170.5$.}
\end{figure}

\medskip
We can see that around 0, the values of the survival function are much better than in the benchmark model: around $20\%$ in the multi-barrier correlation model and around $10\%$ in the benchmark model. However, the values are still low. Indeed, even in the maximal case where the second Brownian motion is the reflection of the first one and the volatilities are equals, the probability for the difference between the Brownian motions to be positive knowing that one starts at $-x, \; x > 0$ and the other at 0 is equal to $2\Phi\left(\frac{-x}{2\sqrt{t}}\right)$ which decreases with $x$.

\subsection{Pricing of European spread options} In this Section, we compare prices of European spread options in the factorial model \eqref{sdecommo} with different structures of dependence: the multi-barrier correlation model (m-b) with correlation equals to 0.3, 0.6, 0.9 and the benchmark model (constant correlation) with correlation equals to 0 and 0.275. Benchmark model with correlation equal to 0 is the same model than multi-barrier correlation model with correlation equal to 0. We price options with payoff $\left(X_t-HY_t\right)^+$, where $X_t$ is an electricity product, $Y_t$ a coal product and H is the conversion factor between electricity and coal. $X_t$ and $Y_t$ are Spot, 1MAH, 3MAH and 6MAH. Parameters used are those of Table \ref{paramgazelec}. The price of the option is equal to $\mathbb{E}\left(\left(X_t-Y_t\right)^+\right)$. We use Monte Carlo to estimate this expectation with a number of simulations equal to 10000. To simulate the processes, we use a step time of 1 hour.

\medskip
Table \ref{0option} gives $95\%$ confidence intervals for the price of spread options with maturity 1 year when $X_0 = HY_0 = 100$. For the multi-barrier correlation model, we choose $\nu = 0$ and $\eta = 0.5$. In the multi-barrier correlation model, the value of the option decreases with the correlation parameters. Indeed, when the correlation parameters increases, the probability to be over 0 is higher, but the values taken by the difference $X_t - HY_t$ are smaller and smaller. The increase in the probability do not compensate the decrease in the values that can be taken and the expectation, i.e. the value of the option decreases. Value of the option in the benchmark model with correlation equal to $0.275$ is close to the one in the multi-barrier correlation model with correlation equal to $0.6$. We also observe that the value of the option decreases with the product maturity, in all the models. 
\begin{footnotesize}
 \begin{table}[h!]
\begin{tabular}{|c|c|c|c|c|c|c|}
\hline
Products / Parameters &	$\rho = 0$	&$\rho = 0.3$, m-b& $\rho = 0.6$, m-b& $\rho = 0.9$, m-b& $\rho = 0.275$, benchmark\\
\hline
Spot &$\left[8.39,8.92\right] $&	$\left[8.44,8.96\right] $&	$\left[7.87,8.37\right] $	&$\left[7.29,7.75\right] $&	$\left[7.69,8.19\right] $\\
\hline
1MAH  & $\left[6.54,6.94\right] $&	$\left[6.56,6.94\right] $	&$\left[5.96,6.30\right] $	&$\left[5.00,5.29\right] $&	$\left[5.80,6.16\right] $	\\
\hline
3MAH    & $\left[5.45,5.78\right] $&	$\left[5.41,5.70\right] $	&$\left[4.79,5.03\right] $	&$\left[3.27,3.41\right] $&	$\left[4.72,5.00\right] $	\\
 \hline
 6MAH    & $\left[5.33,5.69\right] $&	$\left[5.26,5.55\right] $	&$\left[4.65,4.87\right] $	&$\left[3.02,3.15\right] $&	$\left[4.60,4.88\right] $	\\
 \hline
\end{tabular}
 \caption{\label{0option}\it Values of European Spread options $\left(X_t-HY_t\right)^+$ between electricity and coal products in the benchmark model and in the multi-barrier correlation model with parameters $\nu = 0$, $\eta = 0.5$ with $X_0 = HY_0 =100$.} 
 \end{table}
\end{footnotesize}
Table \ref{-20option} gives $95\%$ confidence intervals for the price of spread options with maturity 1 year when $X_0 = 100$ and  $HY_0 = 120$. For the multi-barrier correlation model, we choose $\nu = 170$ and $\eta = 170.5$. Contrarily to results of Table \ref{0option}, the value of the option increases with the correlation parameter in the multi-barrier correlation model. Furthermore, the value of the option in the multi-barrier case is greater than the one of the benchmark model, for all the given correlations. In the constant correlation case, the probability to be greater than 0 is very low. The increase of probability in the multi-barrier correlation model is enough for the option value to be higher. 

\begin{footnotesize}
 \begin{table}[h!]
\begin{tabular}{|c|c|c|c|c|c|c|}
\hline
Products / Parameters &	$\rho = 0$	&$\rho = 0.3$, m-b& $\rho = 0.6$, m-b& $\rho = 0.9$, m-b& $\rho = 0.275$, benchmark\\
\hline
Spot &$\left[2.52,2.83\right] $&	$\left[2.92,3.25\right] $&	$\left[3.03,3.36\right] $	&$\left[3.13,3.48 \right] $&	$\left[2.09 , 2.37\right] $\\
\hline
1MAH  & $\left[1.24,1.42\right] $&	$\left[1.57,1.77\right] $	&$\left[1.72,1.92\right] $	&$\left[1.74,1.98 \right] $&	$\left[0.88 , 1.02\right] $	\\
\hline
3MAH    & $\left[0.67,0.79\right] $&	$\left[0.9,1.02\right] $	&$\left[1.03,1.15\right] $	&$\left[0.81,0.90\right] $&	$\left[0.37,0.45\right] $	\\
 \hline
 6MAH    & $\left[0.63,0.74\right] $&	$\left[0.82,0.94\right] $	&$\left[0.92,1.03\right] $	&$\left[0.67,0.74\right] $&	$\left[0.33,0.41\right] $	\\
 \hline
\end{tabular}
 \caption{\label{-20option}\it Values of European Spread options $\left(X_t-HY_t\right)^+$ between electricity and coal products in the benchmark model and in the multi-barrier correlation model with parameters $\nu = 170$, $\eta = 170.5$ with $X_0 = 100$ and $HY_0 =120$.}  
 \end{table} 
\end{footnotesize}

\section{Proofs}
\label{proofs}

\subsection{Preliminary results}

We start with well known results that will be useful for the proofs of propositions.

\medskip
\begin{lemma} \label{lawsup} \label{lawinf} Let $B = \left(B_t\right)_{t \geq 0}$ be a standard Brownian motion on a filtered probability space $\left(\Omega, \mathcal{F}, \left(\mathcal{F}_t\right)_{t\geq0}, \mathbb{P}\right)$. We have: 
\begin{enumerate}
\item[(i)]  for $y \geq 0$, 
\[\mathbb{P}\Bigl(B_t \leq x,  \underset{s \leq t}{\sup \;} B_s \leq y\Bigr) = \left\lbrace
\begin{array}{lll}
\Phi\Bigl(\frac{x}{\sqrt{t}}\Bigr) - \Phi\Bigl(\frac{x-2y}{\sqrt{t}}\Bigr) & \mbox{if} & x < y  \\
2\Phi\Bigl(\frac{y}{\sqrt{t}}\Bigr) - 1  & \mbox{if} & x \geq y 
\end{array}\right.,  \]
\item[(ii)]  for $y \leq 0$,
\[\mathbb{P}\Bigl(B_t \leq x,  \underset{s \leq t}{\inf \;} B_s \leq y\Bigr) = \left\lbrace
\begin{array}{ccc}
\Phi\Bigl(\frac{x}{\sqrt{t}}\Bigr) & \mbox{if} & x \leq y  \\
2\Phi\Bigl(\frac{y}{\sqrt{t}}\Bigr) - \Phi\Bigl(\frac{-x+2y}{\sqrt{t}}\Bigr)  & \mbox{if} & x > y 
\end{array}\right..  \]
\end{enumerate}
\end{lemma}

\begin{preuve} The reader is referred to \cite[Theorem\ 3.1.1.2, p.\ 137]{jeanblanc09} for the proof of (i) and to \cite[Section\ 3.1.5, p.\ 142]{jeanblanc09} for the proof of (ii).
\end{preuve}

\bigskip
\begin{lemma} \label{stoppingtime} Let $B^1 = \left(B^1_t\right)_{t \geq 0}$ and $B^2 = \left(B^2_t\right)_{t \geq 0}$ be two independent standard Brownian motion defined on a common filtered probability space $\left(\Omega, \mathcal{F}, \left(\mathcal{F}_t\right)_{t\geq0}, \mathbb{P}\right)$with $\left(\mathcal{F}_t\right)_{t\geq0}$ having all the good properties. Let $h \geq 0$ and $\tau^h =  \inf \{t  \geq 0 : B^2_t = h \}$. We have:
\[\mathbb{P}\Bigl(B^1_t - B^1_{\tau^h} \leq x, \tau^h \leq t\Bigr) = \Phi\Bigl(\frac{x-h}{\sqrt{t}}\Bigr){\bf1}_{x < 0} + \Bigl(\Phi\Bigl(\frac{x+h}{\sqrt{t}}\Bigr) - 2\Phi\Bigl(\frac{h}{\sqrt{t}}\Bigr) + 1\Bigr){\bf1}_{x \geq 0}.\]
\end{lemma}

\medskip
\begin{preuve}
Conditional on $\{t \geq \tau^h\}$, $B^1_t - B^1_{\tau^h}$ is a Brownian motion independent to $\mathcal{F}_{\tau^h}$. Then
\[\mathbb{P}\Bigl(B^1_t - B^1_{\tau^h} \leq x, \tau^h \leq t\Bigr) = \mathbb{E}\Bigl(\Phi\Bigl(\frac{x}{\sqrt{t-\tau^h}}\Bigr) {\bf1}_{t \geq \tau^h}\Bigr).\]
The same argument can be used to prove that 
\[\mathbb{P}\Bigl(B^2_t - B^2_{\tau^h} \leq x, \tau^h \leq t\Bigr) = \mathbb{E}\Bigl(\Phi\left(\frac{x}{\sqrt{t-\tau^h}}\right) {\bf1}_{t \geq \tau^h}\Bigr).\]
Then we have 
\begin{align*}
\mathbb{P}\Bigl(B^1_t - B^1_{\tau^h} \leq x, \tau^h \leq t\Bigr) &= \mathbb{P}\Bigl(B^2_t - B^2_{\tau^h} \leq x, \tau \leq t\Bigr)\\
&= \mathbb{P}\Bigl(B^2_t \leq x + h, \underset{s \leq t}{\sup} B^2_s \geq h\Bigr).
\end{align*}
We can conclude using Lemma \ref{lawsup}.
\end{preuve}

\begin{lemma} \label{intphi} Let a, b and x $\in \mathbb{R}$. We have:
\begin{enumerate}
\item[(i)] \[\int_{-\infty}^{x} \Phi\left(a u + b\right) \frac{e^{\frac{-u^2}{2}}}{\sqrt{2\pi}} du = \Phi_{\frac{-a}{\sqrt{a^2+1}}}\Bigl(\frac{b}{\sqrt{a^2+1}}, x\Bigr).\]
\item[(ii)] \[\Phi_{\sqrt{1-\rho^2}}\left(x,y\right) = \Phi\left(y\right)\Phi\Bigl(\frac{x-\sqrt{1-\rho^2}y}{\rho}\Bigr) + \Phi\left(x\right) - \Phi_{\rho}\Bigl(x,\frac{x-\sqrt{1-\rho^2}y}{\rho}\Bigr),\; x, y \in \mathbb{R},\; \rho > 0\]
\item[(iii)]  \[\Phi_{\rho}\left(x,y\right) = \Phi\left(y\right) - \Phi_{-\rho}\left(-x,y\right),\; x, y \in \mathbb{R}\] 
\end{enumerate}
\end{lemma}

\begin{preuve}
The reader is referred to \cite[Proof of Lemma 19, Section 5.3]{deschatre16}.
\end{preuve}

\subsection{Proof of Proposition \ref{survival2}} Let $B^1$ and $Z$ two independent Brownian motion. We consider $B^2 = \rho \tilde{B}^h + \sqrt{1-\rho^2}Z$ with $\tilde{B}^h$ the reflection of $B$ according to the barrier $h$. We have:
\[
\mathbb{P}\left(B^1_t - B^2_t \geq x\right) = \mathbb{P}\Bigl(B^1_t - B^2_t \geq x, \underset{s \leq t}{\sup \;} B^1_s \leq  h\Bigr) + \mathbb{P}\Bigl(B^1_t - B^2_t \geq x, \underset{s \leq t}{\sup \;} B^1_s \geq  h\Bigr)\\
\]
When $\underset{s \leq t}{\sup \;} B^1_s \leq  h$, $B^2_t = -\rho B^1_t + \sqrt{1-\rho^2} Z_t$ and when $\underset{s \leq t}{\sup \;} B^1_s \geq  h$, $B^2_t = \rho B^1_t -2 h \rho+ \sqrt{1-\rho^2} Z_t$. Thus, $\mathbb{P}\left(B^1_t - B^2_t \geq x\right)$ is the sum of the three following terms: 
\begin{enumerate}
\item[(i)] $\mathbb{P}\Bigl(B^1_t \leq \frac{x - 2\rho h + \sqrt{1-\rho^2}  Z_t}{\left(1-\rho\right)} , \underset{s \leq t}{\sup \;} B^1_s \leq  h\Bigr)$,
\item[(ii)] $- \mathbb{P}\Bigl(B^1_t \leq \frac{x + \sqrt{1-\rho^2}  Z_t}{\left(1+\rho\right)}, \underset{s \leq t}{\sup \;} B^1_s \leq  h\Bigr)$,
\item[(iii)] $ \mathbb{P}\Bigl(\left(1-\rho\right) B^1_t - \sqrt{1-\rho^2}  Z_t \geq x - 2\rho h\Bigr)$.
\end{enumerate}
Since $B^1$ and $Z$ are independent, (i) is equal to the sum of the three following terms:
\begin{equation} \label{exp1}
\mathbb{E}\Bigl( \Phi\Bigl(\frac{x -2\rho h + \sqrt{1-\rho^2}  Z_t}{\left(1-\rho\right)\sqrt{t}} \Bigr){\bf 1}_{Z_t \leq \frac{h\left(1+ \rho\right)-x}{\sqrt{1-\rho^2}}}\Bigr),
\end{equation}
\begin{equation} \label{exp2}
 - \mathbb{E}\Bigl(\Phi\Bigl( \frac{x -2h + \sqrt{1-\rho^2}  Z_t }{\left(1-\rho\right)\sqrt{t}} \Bigr) {\bf 1}_{Z_t \leq \frac{h\left(1+\rho\right)-x}{\sqrt{1-\rho^2}}}\Bigr) 
\end{equation} 
and
\begin{equation} \label{exp3}
\mathbb{E}\Bigl(\Bigl(2\Phi\Bigl( \frac{h}{\sqrt{t}} \Bigr) -1 \Bigr){\bf 1}_{Z_t \geq \frac{h\left(1+\rho\right)-x}{\sqrt{1-\rho^2}}}\Bigr)
\end{equation}
with the use of Lemma \ref{lawsup}. According to Lemma \ref{intphi} (i), \eqref{exp1} is equal to
\begin{equation} \label{1termnonsimplify} \Phi_{-\sqrt{\frac{1+\rho}{2}}}\Bigl(\frac{x - 2\rho h}{\sqrt{2\left(1-\rho\right)t}}, \frac{h\left(1+\rho\right)-x}{\sqrt{\left(1-\rho^2\right)t}}\Bigr).
\end{equation}
Using Lemma \ref{intphi} (ii), we find that the first term of (i) \eqref{1termnonsimplify} is equal to 
\begin{equation} \label{(i)1term}
-\Phi_{\sqrt{\frac{1+\rho}{2}}}\Bigl(\frac{-x+2 \rho h}{\sqrt{2\left(1-\rho\right)t}}, \frac{h\left(1+\rho\right)-x}{\sqrt{\left(1-\rho^2\right)t}}\Bigr) + \Phi\Bigl( \frac{h\left(1+\rho\right)-x}{\sqrt{\left(1-\rho^2\right)t}}\Bigr).
\end{equation}
In the same way, the second term of (i) \eqref{exp2} is equal to: 
\begin{equation} \label{(i)2term}\Phi_{\sqrt{\frac{1+\rho}{2}}}\Bigl(\frac{-x+2h}{\sqrt{2\left(1-\rho\right)t}}, \frac{h\left(1+\rho\right)-x}{\sqrt{\left(1-\rho^2\right)t}}\Bigr) - \Phi\Bigl( \frac{h\left(1+\rho\right)-x}{\sqrt{\left(1-\rho^2\right)t}}\Bigr). 
\end{equation}
The last one \eqref{exp3} is equal to 
\begin{equation} \label{(i)3term}
\Bigl(2\Phi\Bigl( \frac{h}{\sqrt{t}} \Bigr) -1 \Bigr)\Phi\Bigl( \frac{x-h\left(1+\rho\right)}{\sqrt{\left(1-\rho^2\right)t}}\Bigr).
\end{equation}
Using the same scheme of proof that for (i), we find that (ii) is equal to the sum of the three following terms:
\begin{equation} \label{(ii)1term1}
\Phi_{\sqrt{\frac{1-\rho}{2}}}\Bigl(\frac{-x}{\sqrt{2\left(1+\rho\right)t}}, \frac{h\left(1+\rho\right)-x}{\sqrt{\left(1-\rho^2\right)t}}\Bigr) - \Phi\Bigl(\frac{h\left(1+\rho\right)-x}{\sqrt{\left(1-\rho^2\right)t}}\Bigr),
\end{equation}
\begin{equation}
\label{(ii)2term1}-\Phi_{\sqrt{\frac{1-\rho}{2}}}\Bigl(\frac{-x+2h\left(1+\rho\right)}{\sqrt{2\left(1+\rho\right)t}}, \frac{h\left(1+\rho\right)-x}{\sqrt{\left(1-\rho^2\right)t}}\Bigr) + \Phi\Bigl(\frac{h\left(1+\rho\right)-x}{\sqrt{\left(1-\rho^2\right)t}}\Bigr)
\end{equation}
and
\begin{equation} \label{(ii)3term}
-\Bigl(2\Phi\Bigl( \frac{h}{\sqrt{t}} \Bigr) -1 \Bigr)\Phi\Bigl( \frac{x-h\left(1+\rho\right)}{\sqrt{\left(1-\rho^2\right)t}}\Bigr).
\end{equation}
Using Lemma \ref{intphi} (iii), we find that \eqref{(ii)1term1} is equal to
\begin{equation} \label{(ii)1term}
\Phi\Bigl(\frac{-x}{\sqrt{2\left(1+\rho\right)t}}\Bigr) \Phi\Bigl(\frac{-x+2 h}{\sqrt{2\left(1-\rho\right)t}}\Bigr) -\Phi_{\sqrt{\frac{1+\rho}{2}}}\Bigl(\frac{-x+2h}{\sqrt{2\left(1-\rho\right)t}}, \frac{h\left(1+\rho\right)-x}{\sqrt{\left(1-\rho^2\right)t}}\Bigr) .
\end{equation}
and that \eqref{(ii)2term1} to 
\begin{equation} \label{(ii)2term}
- \Phi\Bigl(\frac{-x+2h\left(1+\rho\right)}{\sqrt{2\left(1+\rho\right)t}}\Bigr) \Phi\Bigl(\frac{-x+2\rho h}{\sqrt{2\left(1+\rho\right)t}}\Bigr) + \Phi_{\sqrt{\frac{1+\rho}{2}}}\Bigl(\frac{-x+2\rho h}{\sqrt{2\left(1-\rho\right)t}},   \frac{h\left(1+\rho\right)-x}{\sqrt{\left(1-\rho^2\right)t}}\Bigr).
\end{equation}
Finally, we have (iii) equal to 
\begin{equation}\label{(iii)term}
\Phi\Bigl(\frac{-x+2\rho h}{\sqrt{2\left(1-\rho\right)t}}\Bigr).
\end{equation}
$\mathbb{P}\left(B^1_t - B^2_t \geq x\right)$ is the sum of \eqref{(i)1term}, \eqref{(i)2term}, \eqref{(i)3term}, \eqref{(ii)1term}, \eqref{(ii)2term}, \eqref{(ii)3term} and \eqref{(iii)term}.

\subsection{Proof of Proposition \ref{model}}

\noindent (i) This part of the proof can be done by induction.

\smallskip
\noindent (ii) For $\tau_{0} = 0  \leq t \leq \tau_{1}$, $X_t - Y^n_t =  \left(1+\rho\right)B^X_t - \sqrt{1-\rho^2}B^Y_t$. The equality holds for k = 0.

\smallskip
Let us suppose that the property true at rank $k < n + 1$, that is 
\[X_t - Y^n_t = \left(1+\left(-1\right)^k\rho\right)\left(B^X_t - B^X_{\tau_{k}}\right) - \sqrt{1-\rho^2}\left(B^Y_t - B^Y_{\tau_{k}}\right) + \alpha_k, \quad \tau_{k}  \leq t \leq \tau_{k+1}. \]
If $\tau_{k}  \leq t \leq \tau_{k+1}$, $Y^n_t = \rho\tilde{B}^k_t + \sqrt{1-\rho^2} B^Y_t$ then 
\begin{equation} \label{spreadvaluepart1} X_t - \rho\tilde{B}^k_t - \sqrt{1-\rho^2} B^Y_t = \left(1+\left(-1\right)^k\rho\right)\left(B^X_t - B^X_{\tau_{k}}\right) - \sqrt{1-\rho^2}\left(B^Y_t - B^Y_{\tau_{k}}\right) + \alpha_k. 
\end{equation}
As $\tilde{B}^k_t$ does not change after time $\tau_{k+1}$, this relationship remains true for all time greater than $\tau_{k}$.

\noindent At time $\tau_{k+1}$, we have the equation
\begin{equation}  \label{spreadvaluepart2} \alpha_{k+1} = \left(1+\left(-1\right)^k\rho\right)\left(B^X_{\tau_{k+1}} - B^X_{\tau_{k}}\right) - \sqrt{1-\rho^2}\left(B^Y_{\tau_{k+1}} - B^Y_{\tau_{k}}\right) + \alpha_k.
\end{equation}
Taking the difference between Equation \eqref{spreadvaluepart1} and Equation \eqref{spreadvaluepart2}, we have 
\begin{equation*} X_t - \rho\tilde{B}^k_t - \sqrt{1-\rho^2} B^Y_t = \left(1+\left(-1\right)^k\rho\right)\left(B^X_t - B^X_{\tau_{k+1}}\right) - \sqrt{1-\rho^2}\left(B^Y_t - B^Y_{\tau_{k+1}}\right) + \alpha_{k+1}.
\end{equation*}
Let $\tau_{k+1}  \leq t \leq \tau_{k+2}$. If $k = n$, the proof is over because $Y^n_t = \rho\tilde{B}^n_t + \sqrt{1-\rho^2} B^Y_t$ for $\tau_{n+1}$. Otherwise, $Y^n_t = \rho \tilde{B}^{k+1}_t + \sqrt{1-\rho^2}B^Y_{t}$ with $\tilde{B}^{k+1}_t = \mathcal{R}\left(\tilde{B}^{k}_t, {\tau_{k+1}}\right) = 2\tilde{B}^{k}_{\tau_{k+1}} - \tilde{B}^{k}_t$ and 
\begin{align*} 
X_t - Y^n_t &= X_t - \rho \tilde{B}^{k+1}_t - \sqrt{1-\rho^2}B^Y_{t}\\
&= X_t - \rho \tilde{B}^{k}_t - \sqrt{1-\rho^2}B^Y_{t} + \rho(\tilde{B}^{k}_{t} - \tilde{B}^{k+1}_{t})\\
&= X_t - \rho \tilde{B}^{k}_t - \sqrt{1-\rho^2}B^Y_{t} + 2\rho(\tilde{B}^{k}_{t} - \tilde{B}^{k}_{\tau_{k+1}})\\
&= \left(1+\left(-1\right)^k\rho\right)\left(B^X_t - B^X_{\tau_{k+1}}\right) - \sqrt{1-\rho^2}\left(B^Y_t - B^Y_{\tau_{k+1}}\right) + \alpha_{k+1} + 2\rho\left(\tilde{B}^{k}_{t} - \tilde{B}^{k}_{\tau_{k+1}}\right).
\end{align*}
Let $s, t > \tau_k$, we have 
\begin{align*}
\tilde{B}^k_t - \tilde{B}^k_s &= -\tilde{B}^{k-1}_t + 2\tilde{B}^{k-1}_{\tau_{k}} + \tilde{B}^{k-1}_s - 2\tilde{B}^{k-1}_{\tau_{k}} \\
&= - \left(\tilde{B}^{k-1}_t - \tilde{B}^{k-1}_s\right) = \left(-1\right)^k(\tilde{B}^0_t-\tilde{B}^{0}_s)\\ &
= \left(-1\right)^{k+1}\left(B^X_t - B^X_s\right).
\end{align*}
Then $2\rho\left(\tilde{B}^{k}_{t} - \tilde{B}^{k}_{\tau_{k+1}}\right) = 2\rho\left(-1\right)^{k+1}\left(B^X_t - B^X_{\tau_{k+1}}\right)$ and we find that the property holds at rank $k+1$, which achieves the proof.

\smallskip
\noindent (iii) We first need Lemma \ref{lawtaumodel}. 

\begin{lemma} \label{lawtaumodel}
We have \begin{equation}
\label{lawtau}
\tau_{k}  \overset{d}{=}  \inf\{t \geq 0 : B_t = u_k\}
\end{equation}
where
\begin{equation} 
\left\lbrace
\begin{array}{l}
u_0 = 0\\
u_{k} = \frac{\eta}{\sqrt{2\left(1+\rho\right)}} + \frac{\left(\eta - \nu\right)}{\sqrt{2}}\Bigl(\frac{\lfloor \frac{k}{2} \rfloor}{\sqrt{1-\rho}} + \frac{\lfloor \frac{k-1}{2} \rfloor}{\sqrt{1+\rho}}\Bigr) \quad k \geq 1  \\
\end{array}\right. 
\end{equation}
with $B$ a standard Brownian motion and $\lfloor . \rfloor$ the floor function.
\end{lemma}

\begin{proof}The property holds for $k = 1$.
\medskip
Let us suppose that the property holds for $k = 2p + 1$. $X_{\tau_{k}} - Y^n_{\tau_{k}} = \eta$ and $\tau_{k+1}$ is the first time greater than $\tau_k$ when $X_t - Y^n_t$ goes to $\nu$. According to Equation \eqref{valuespread}, 
\[\mathbb{P}\left(\tau_{k+1} \leq t\right) =  \mathbb{P}\Bigl(\underset{\tau_{k} \leq s \leq t}{\inf} \left(1-\rho\right)\left(B^X_s - B^X_{\tau_{k}}\right) - \sqrt{1-\rho^2}\left(B^Y_s - B^Y_{\tau_{k}}\right) + \eta \leq \nu, t \geq {\tau_{k}}\Bigr).\]
If $t \geq \tau_{k}$, $(B^X_t - B^X_{\tau_{k}})$ and $(B^Y_t - B^Y_{\tau_{k}})$ are Brownian motions independent of $\mathcal{F}_{\tau_{k}}$. Then using Lemma \ref{lawinf} and Lemma \ref{stoppingtime}, we have
\begin{align*}
\mathbb{P}\left(\tau_{k+1} \leq t\right) &= \mathbb{E}\Bigl(2 \Phi\Bigl(\frac{\nu-\eta}{\sqrt{2\left(1-\rho\right)\left(t - \tau_{k}\right)}}\Bigr) {\bf1}_{t \geq \tau_{k}}\Bigr) \\
& = 2 \mathbb{P}\left(\left(1-\rho\right)\left(B^X_t - B^X_{\tau_{k}}\right) - \sqrt{1-\rho^2}\left(B^Y_t - B^Y_{\tau_{k}}\right) \leq \nu - \eta, t \geq \tau_{k}\right)\\
& = 2 \Phi\Bigl(\frac{\nu - \eta}{\sqrt{2\left(1-\rho\right)t}} - u_k\Bigr).
\end{align*}
This is the law of the stopping time  $\tau = \inf\{t \geq 0 : B_t = u_k + \frac{\eta - \nu}{\sqrt{2(1-\rho)}}\}$ and the property holds for $k+1$. The proof is similar for $k = 2p$.
\end{proof}

The proof of (iii) can be done. $\{N_t = n \} = \{\tau_n \leq t, \tau_{n+1} > t\}$ and then we have
\begin{align*}
\mathbb{E}\left(N_t\right) = \sum_{n = 1}^{\infty} n  \mathbb{P}\left(\tau_n \leq t, \tau_{n+1} > t \right)
& \leq \sum_{n = 1}^{\infty} n \mathbb{P}\left(\tau_n \leq t \right) 
\end{align*}
According to Lemma \ref{lawtaumodel}, $\mathbb{P}\left(\tau_n \leq t \right)  = 2\int_{\frac{u_n}{\sqrt{t}}}^{\infty} \frac{e^{\frac{-y^2}{2}}}{\sqrt{2\pi}} dy = 2\Phi\left(\frac{-u_n}{\sqrt{t}}\right)$. Since $\lim \limits_{n \to \infty} u_n = \infty$ and $n = \underset{n \to \infty}{O}\left(u_n\right)$,
\[
\mathbb{P}\left(\tau_n \leq t \right) = \underset{n \to \infty}{o}\Bigl(e^{\frac{-u_n^2}{2t}}\Bigr) = \underset{n \to \infty}{o}\Bigl(\frac{1}{u_n^3}\Bigr) = \underset{n \to \infty}{O}\Bigl(\frac{1}{n^3}\Bigr). 
\]
Then $n\mathbb{P}\left(\tau_n \leq t \right) = \underset{n \to \infty}{O}\left(\frac{1}{n^2}\right)$ and $\mathbb{E}\left(N_t\right) < \infty$ by comparison theorem of positive series, implying $N_t < \infty$ almost surely.

\smallskip
\noindent (iv) If $n \geq N_t$, the number of reflections of $X-Y^n$ between time $0$ and time $t$ is equal to $N_t$ and $Y^{N_t}_t = Y^n_t = $ almost surely.

\subsection{Proof of Proposition \ref{spreaddistrib} and Corollary \ref{convergencesurvival}}

We start with Lemma \ref{lawspreadtau}.
\begin{lemma} \label{lawspreadtau} For $t > 0$, $x \in \mathbb{R}$, 
\[
\mathbb{P}\left(X_t - Y^{n}_t \leq x, t \geq \tau_{n}\right) = 
\left \lbrace
\begin{array}{ccc}
\Phi\Bigl(\frac{x - \alpha_n}{\sqrt{2\left(1+\left(-1\right)^n\rho\right)t}} - \frac{u_{n}}{\sqrt{t}}\Bigr) & \mbox{if} & x < \alpha_n \\
\Phi\Bigl(\frac{x - \alpha_n}{\sqrt{2\left(1+\left(-1\right)^n\rho\right)t}} + \frac{u_{n}}{\sqrt{t}}\Bigr) - 2\Phi\left(\frac{u_{n}}{\sqrt{t}}\right) + 1 & \mbox{if} & x \geq \alpha_n \\
\end{array}\right.
\]
and 
\[
\mathbb{P}\left(X_t - Y^{n}_t \leq x, t \geq \tau_{n+1}\right)  = 
\left\lbrace
\begin{array}{ccc}
 \Phi\Bigl(\frac{x - \alpha_{n+1}}{\sqrt{2\left(1+(-1)^n\rho\right)t}} - \frac{u_{n+1}}{\sqrt{t}}\Bigr) & \mbox{if} & x < \alpha_{n+1} \\
 \Phi\Bigl(\frac{x - \alpha_{n+1}}{\sqrt{2\left(1+\left(-1\right)^n\rho\right)t}} + \frac{u_{n+1}}{\sqrt{t}}\Bigr) - 2\Phi\Bigl(\frac{u_{n+1}}{\sqrt{t}}\Bigr) + 1& \mbox{if} & x \geq \alpha_{n+1} \\
\end{array}\right.. 
\]
\end{lemma}

\medskip
\begin{preuve}
We have:
\begin{align*}
\mathbb{P}\left(X_t - Y^{n}_t \leq x, t \geq \tau_{n}\right) &= \mathbb{P}\left(\left(1+\left(-1\right)^n\rho\right)\left(B^X_t - B^X_{\tau_{n}}\right) - \sqrt{1-\rho^2}\left(B^Y_t - B^Y_{\tau_{n}}\right) + \alpha_n \leq x, t \geq \tau_{n}\right) \\
&= \mathbb{E}\Bigl(\Phi\Bigl(\frac{x-\alpha_n}{\sqrt{2\left(1+\left(-1\right)^n\rho\right)\left(t-\tau_{n}\right)}}\Bigr) {\bf1}_{t \geq \tau_{n}}\Bigr).
\end{align*}
However, according to Equation \eqref{lawtau}, $\tau_{n} \sim \tau' = \inf\{t \geq 0 : B_t = u_n\}$ where $B_t$ is a standard Brownian motion. Then we have, using Lemma \ref{stoppingtime},

\begin{align*}
\mathbb{P}\left(X_t - Y^{n}_t \leq x, t \geq \tau_{n}\right) &= \mathbb{E}\Bigl(\Phi\Bigl(\frac{x-\alpha_n}{\sqrt{2\left(1+\left(-1\right)^n\rho\right)\left(t-\tau'\right)}}\Bigr) {\bf1}_{t \geq \tau'}\Bigr)\\
& =  \Phi\Bigl(\frac{x - \alpha_n}{\sqrt{2\left(1+\left(-1\right)^n\rho\right)t}} - \frac{u_{n}}{\sqrt{t}}\Bigr){\bf1}_{x < \alpha_n} \\
&+ \Bigl(\Phi\Bigl(\frac{x - \alpha_n}{\sqrt{2\left(1+\left(-1\right)^n\rho\right)t}} - \frac{u_{n}}{\sqrt{t}}\Bigr) - 2\Phi\left(\frac{u_{n}}{\sqrt{t}}\right) + 1\Bigr){\bf1}_{x \geq \alpha_n}.
\end{align*}

\medskip
The proof is the same for $\mathbb{P}\left(X_t - Y^{n}_t \leq x, t \geq \tau_{n+1}\right)$.

\end{preuve}

We can now prove Proposition \ref{spreaddistrib}. We have:
\[
\mathbb{P}\left(X_t - Y^{n+1}_t \geq x\right) - \mathbb{P}\left(X_t - Y^{n}_t \geq x\right) = \mathbb{P}\left(X_t - Y^n_t \leq x\right) - \mathbb{P}\left(X_t - Y^{n+1}_t \leq x\right)\]
which is equal to 
\begin{align*}
&\mathbb{P}\left(X_t - Y^n_t \leq x, \tau_{n+1} \leq t\right) - \mathbb{P}\left(X_t - Y^{n+1}_t \leq x, \tau_{n+1} \leq t\right) \\
&+ \mathbb{P}\left(X_t - Y^n_t \leq x, \tau_{n+1} \geq t\right) - \mathbb{P}\left(X_t - Y^{n+1}_t \leq x, \tau_{n+1} \leq t\right).
\end{align*}
For $\tau_{n+1} \geq t$, $X_t - Y^n_t$ and $X_t - Y^{n+1}_t$ are equals then \[\mathbb{P}\left(X_t - Y^n_t \leq x, \tau_{n+1} \geq t\right) = \mathbb{P}\left(X_t - Y^{n+1}_t \leq x, \tau_{n+1} \geq t\right).\] 
We then have
\[\mathbb{P}\left(X_t - Y^{n+1}_t \geq x\right) - \mathbb{P}\left(X_t - Y^{n}_t \geq x\right) = \mathbb{P}\left(X_t - Y^n_t \leq x, \tau_{n+1} \leq t\right) - \mathbb{P}\left(X_t - Y^{n+1}_t \leq x, \tau_{n+1} \leq t\right)\]
and  we can conclude using Lemma \ref{lawspreadtau}.

\smallskip
Since for $n \geq N_t$ $X_t - Y^n_t = X_t - Y_t^{N_t}$, $X_t - Y_t^{N_t}$ is the limit in law of $X_t - Y^n_t$, and $\mathbb{P}\left(X_t - Y_t \geq x \right) = \lim \limits_{n \to \infty} \mathbb{P}\left(X_t - Y^n_t \geq x\right)$.

\medskip
The proof for Corollary \ref{convergencesurvival} follows. 

\smallskip
Let $x \in \left[\nu, \eta\right]$ and let assume $\rho > 0$. We have:
\[\mathbb{P}\left(X_t - Y^{n+1}_t \geq x\right) - \mathbb{P}\left(X_t - Y^{n}_t \geq x\right) = p_{n+1}\left(t,x\right).\]
If $n$ is even, 
\[ p_{n+1}\left(t, x\right) = \Phi\Bigl(\frac{x-\eta}{\sqrt{2\left(1+\rho\right)t}} - \frac{u_{n+1}}{\sqrt{t}}\Bigr) - \Phi\Bigl(\frac{x-\eta}{\sqrt{2\left(1-\rho\right)t}} - \frac{u_{n+1}}{\sqrt{t}}\Bigr) > 0. \]
If $n$ is odd, 
\[p_{n+1}\left(t, x\right) = \Phi\Bigl(\frac{x-\nu}{\sqrt{2\left(1-\rho\right)t}} + \frac{u_{n+1}}{\sqrt{t}}\Bigr) - \Phi\Bigl(\frac{x-\nu}{\sqrt{2\left(1+\rho\right)t}} + \frac{u_{n+1}}{\sqrt{t}}\Bigr) > 0, \]
which achieves the proof. 

\subsection{Proof of Proposition \ref{solutionsde}}

As $\tilde{\rho}$ is Lipschitz and $\underset{x \in \mathbb{R}}{\sup} |\tilde{\rho}\left(x\right)| < 1$, $\sqrt{1-\tilde{\rho}^2}$ is Lipschitz on $\mathbb{R}$ and \\ $\left(x,y\right) \mapsto \begin{pmatrix}
   1 & 0 \\
   \tilde{\rho}\left(x-y\right) & \sqrt{1-\tilde{\rho}\left(x-y\right)^2}
\end{pmatrix}$
is Lipschitz on $\mathbb{R}^2$, which is a sufficient condition for the system to have a strong solution. This solution is Markovian. 

\smallskip
$X$ is clearly a Brownian motion. By the L\'evy characterization of the Brownian motion, $Y$ is also a Brownian motion.

\vip

\paragraph{\bf Acknowledgements.} I am grateful to Olivier F\'eron and Marc Hoffmann for helpful discussion and comments. This research is supported by the department OSIRIS (Optimization, SImulation, RIsk and Statistics for Energy Markets) of EDF in the context of a CIFRE contract and by FiME (Finance for Energy Markets) Research Initiative. I thank the referees for valuable comments improving the paper considerably.



\begin{thebibliography}{10}

\bibitem{aid09}
Aid, R., L. Campi, A. Nguyen Huu, and N. Touzi (2009).
\newblock A structural risk-neutral model of electricity prices.
\newblock {\em Int. J. Theor. Appl. Finan.}
  12(07), 925-947.

\bibitem{aid13}
Aid, R., L. Campi, and N. Langren{\'e} (2013).
\newblock A structural risk-neutral model for pricing and hedging power
  derivatives.
\newblock {\em Math. Finance} 23(3), 387-438.

\bibitem{benth08}
Benth, F. E. and S. Koekebakker (2008).
\newblock Stochastic modeling of financial electricity contracts.
\newblock {\em Energy Econ.} 30(3), 1116-1157.

\bibitem{bosc12}
Bosc, D. (2012).
\newblock {\em Three essays on modeling the dependence between financial
  assets}.
\newblock PhD thesis, Ecole Polytechnique X.

\bibitem{carmona14}
Carmona, R. and M. Coulon (2014).
\newblock A survey of commodity markets and structural models for electricity
  prices.
\newblock In {\em Quantitative Energy Finance}, pp. 41-83. Springer.

\bibitem{carmona03}
Carmona, R. and V. Durrleman (2003).
\newblock Pricing and hedging spread options.
\newblock {\em SIAM Rev.} 45(4), 627-685.

\bibitem{cherubini04}
Cherubini, U., E. Luciano, and W. Vecchiato (2004).
\newblock {\em Copula methods in finance}.
\newblock John Wiley \& Sons.

\bibitem{deschatre16}
Deschatre, T. (2016).
\newblock On the control of the difference between two brownian motions: a
  dynamic copula approach.
\newblock {\em Depend. Model.}.

\bibitem{dupire94}
Dupire, B. (1994).
\newblock Pricing with a smile.
\newblock {\em Risk} 7(1), 18-20.

\bibitem{feron15}
F\'eron, O. and E. Daboussi (2015).
\newblock {\em Commodities, Energy and Environmental Finance}, chapter
  Calibration of electricity price models, pp. 183-207.
\newblock Springer.

\bibitem{gourieroux09}
Gouri{\'e}roux, C., J. Jasiak, and R. Sufana (2009).
\newblock The wishart autoregressive process of multivariate stochastic
  volatility.
\newblock {\em J. Econometrics} 150(2), 167-181.

\bibitem{heath92}
Heath, D., R. Jarrow, and A. Morton (1992).
\newblock Bond pricing and the term structure of interest rates: A new
  methodology for contingent claims valuation.
\newblock {\em Econometrica}, 77-105.

\bibitem{heston93}
Heston, S. (1993).
\newblock A closed-form solution for options with stochastic volatility with
  applications to bond and currency options.
\newblock {\em Rev. Financ. Stud.} 6(2), 327-343.

\bibitem{jaworski13}
Jaworski, P. and M. Krzywda (2013).
\newblock Coupling of wiener processes by using copulas.
\newblock {\em Statist. Probab. Lett.} 83(9), 2027-2033.

\bibitem{jeanblanc09}
Jeanblanc, M., M. Yor, and M. Chesney (2009).
\newblock {\em Mathematical methods for financial markets}.
\newblock Springer.

\bibitem{langnau10}
Langnau, A. (2010).
\newblock A dynamic model for correlation.
\newblock {\em Risk} 23(4), 74.

\bibitem{nakajima12}
Nakajima, K. and K. Ohashi (2012).
\newblock A cointegrated commodity pricing model.
\newblock {\em J. Futures Markets} 32(11), 995-1033.

\bibitem{samuelson65}
Samuelson, P. A. (1965).
\newblock Proof that properly anticipated prices fluctuate randomly.
\newblock {\em IMR} 6(2), 41.

\bibitem{sklar59}
Sklar, M. (1959).
\newblock {\em Fonctions de r{\'e}partition {\`a} n dimensions et leurs
  marges}.
\newblock Universit{\'e} Paris 8.

\bibitem{vasicek77}
Vasicek, O. (1977).
\newblock An equilibrium characterization of the term structure.
\newblock {\em J. Finan. Econ.} 5(2), 177-188.

\end{thebibliography}
\end{document}